\documentclass[a4paper, 10pt]{amsart}

\usepackage[utf8]{inputenc}
\usepackage{amsmath, amssymb, amsfonts, mathtools}
\usepackage{mathrsfs}

\usepackage[dvips,%
    includehead,%
    includefoot,%
    nomarginpar,%
    lmargin=1.3in,%
    rmargin=1.3in,%
    tmargin=1.4in,%
    bmargin=1.4in,%
  ]{geometry}
\usepackage{caption}
\usepackage[labelformat=simple]{subcaption}

\usepackage{xcolor}
\definecolor{lightblue}{rgb}{.1,0.35,.8}
\usepackage[colorlinks,
    linkcolor={lightblue},
    citecolor={lightblue},
    urlcolor={black}]{hyperref}

\usepackage{enumerate,enumitem}
\usepackage{tikz}
\usetikzlibrary{calc, matrix, arrows, cd, decorations.markings, knots}
\usetikzlibrary{decorations.pathmorphing}
\usetikzlibrary{decorations.pathreplacing}

\usepackage{adjustbox,expl3,etoolbox}
\letcs\replicate{prg_replicate:nn}

\theoremstyle{plain}
\newtheorem{proposition}{Proposition}[section]
\newtheorem{theorem}{Theorem}[section]
\newtheorem{corollary}{Corollary}[section]
\newtheorem{lemma}{Lemma}[section]

\theoremstyle{definition}
\newtheorem{definition}{Definition}[section]

\theoremstyle{remark}
\newtheorem{remark}{Remark}[section]

\makeatletter
\let\c@corollary=\c@theorem
\let\c@proposition=\c@theorem
\let\c@lemma=\c@theorem
\let\c@remark=\c@theorem
\let\c@definition=\c@theorem
\let\c@notation=\c@theorem
\let\c@construction=\c@theorem
\let\c@example=\c@theorem
\let\c@equation\c@theorem
\let\c@question\c@theorem
\makeatother

\newcommand{\wt}{\widetilde}

\newcommand{\ol}{\overline}
\newcommand{\sm}{\setminus}

\newcommand{\Q}{\mathbb{Q}}
\newcommand{\Z}{\mathbb{Z}}

\newcommand{\C}{\mathbb{C}}
\newcommand{\R}{\mathbb{R}}

\newcommand{\CP}{\operatorname{\C P}}
\newcommand{\RP}{\operatorname{\R P}}

\newcommand{\Id}{\operatorname{Id}}

\newcommand{\Ext}{\operatorname{Ext}}

\newcommand{\Hom}{\operatorname{Hom}}

\newcommand{\coker}{\operatorname{coker}}
\newcommand{\BTOPSpin}{\operatorname{BTOPSpin}}

\newcommand{\TOPSpin}{\operatorname{TOPSpin}}

\newcommand{\BSTOP}{\operatorname{BSTOP}}

\DeclareMathOperator\pt{pt}
\DeclareMathOperator\pr{pr}

\DeclareMathOperator\Sq{Sq}

\DeclareMathOperator\Wh{Wh}

\renewcommand{\epsilon}{\varepsilon}
\renewcommand{\phi}{\varphi}

\begin{document}
\title{Simple spines of homotopy $2$-spheres are unique}

\author[Patrick Orson]{Patrick Orson}
\address{Department of Mathematics, California Polytechnic State University, USA}
\email{porson@calpoly.edu}

\author[Mark Powell]{Mark Powell}
\address{School of Mathematics and Statistics, University of Glasgow, UK}
\email{mark.powell@glasgow.ac.uk}

\def\subjclassname{\textup{2020} Mathematics Subject Classification}
\expandafter\let\csname subjclassname@1991\endcsname=\subjclassname

\begin{abstract}
A locally flatly embedded $2$-sphere in a compact $4$-manifold $X$ is called a spine if the inclusion map is a homotopy equivalence. A spine is called simple if the complement of the $2$-sphere has abelian fundamental group. We prove that if two simple spines represent the same generator of $H_2(X)$ then they are ambiently isotopic. In particular, the theorem applies to simple shake-slicing $2$-spheres in knot traces.
\end{abstract}

\maketitle

\section{Introduction}

In this article, unless otherwise specified, we work in the category of topological manifolds, and embeddings are assumed to be locally flat. Let $X$ be a compact 4-manifold homotopy equivalent to $S^2$. A \emph{spine} of $X$ is an embedded, oriented sphere $S \subseteq X$ such that $[S] \in H_2(X) \cong \Z$ is a generator. A spine is called~\emph{simple}   if    $\pi_1(X\sm S)$ is abelian.

It is straightforward to build many interesting compact $4$-manifolds homotopy equivalent to~$S^2$. Indeed, given a knot $K\subseteq S^3$ and an integer~$n$, the \emph{knot trace} $X_n(K)$, formed by attaching a 2-handle $D^2 \times D^2$ to the boundary of the $4$-ball, with attaching circle $K$ and framing~$n$, is an example of a $4$-manifold homotopy equivalent to~$S^2$.
Together with Feller, Miller, Nagel, and Ray, we gave algebraic criteria which hold if and only if $X_n(K)$ admits a simple spine~\cite{FMNOPR}. Our main theorem gives the corresponding uniqueness result. Moreover, our uniqueness result holds not only for knot traces, but for arbitrary homotopy 2-spheres.

\begin{theorem}\label{theoremA}
  Let $X\simeq S^2$ be a compact 4-manifold. Let $S_0$ and $S_1$ be simple spines with $[S_0]=[S_1] \in H_2(X)$ a generator.  Then $S_0$ and $S_1$ are ambiently isotopic
  % rel.\ boundary.
 via an isotopy that restricts to the identity on~$\partial X$.
\end{theorem}

To prove Theorem~\ref{theoremA} we first show, using the methods of modified surgery theory~\cite{MR1709301}, that the spheres $S_0$ and $S_1$ have homeomorphic exteriors rel.\ boundary. This implies there is an orientation preserving homeomorphism of pairs $(X,S_0) \cong (X,S_1)$.  We then deduce ambient isotopy by applying work from ~\cite{OP-mcgs} to show that this homeomorphism is isotopic rel.\ boundary to the identity. In \cite{OP-mcgs} we computed the topological mapping class group of compact, simply connected 4-manifolds with nonempty boundary.  The relevant part of that computation for the present paper is the following.

\begin{theorem}[{\cite[Corollary~C]{OP-mcgs}}]\label{corollaryC}
  Let $X$ be a compact, simply connected 4-manifold such that $\partial X$ is connected and has the rational homology of either $S^3$ or $S^1\times S^2$. Let $F \colon X \to X$ be a homeomorphism that restricts to the identity on $\partial X$ and is such that $F_* = \Id \colon H_2(X) \to H_2(X)$. Then $F$ is topologically isotopic rel.\ boundary to the identity map of~$X$.
\end{theorem}

Theorem~\ref{theoremA} can be compared to other topological uniqueness results for surfaces embedded in $4$-manifolds. The earliest example is the theorem of Freedman and Quinn~\cite[Theorem~11.7A]{FQ}, which states that any pair of embedded $2$-spheres $S_0,S_1\subseteq S^4$ with $\pi_1(S^4\sm S_i)\cong\Z$ must be isotopic.
Lee-Wilczy\'{n}ski~\cite[Theorem~1.2~and~Corollary~1.3]{lee-wilczynski} and Hambleton-Kreck~\cite[Theorems~4.5~and~4.8]{hambleton-kreck:1993} extended this to give conditions implying homotopic, simple embeddings of 2-spheres in arbitrary closed, simply connected, topological 4-manifolds are topologically isotopic. Results of a similar nature for higher genus surfaces were proven by Sunukjian in~\cite[\textsection 7]{Sunukjian-IMRN}.

More recently, the second named author and Conway proved uniqueness results for slice discs in $D^4$ whose complements have fundamental group $\Z$ or $\Z\ltimes\Z[1/2]$~\cite{MR4300918}. In~\cite{CP}, this was extended to higher genus surfaces in $D^4$ whose complements have fundamental group $\Z$.
These previous isotopy uniqueness results for embedded surfaces in $4$-manifolds with nonempty boundary were restricted to studying the ambient $4$-manifold $D^4$. The reason is that in $D^4$ the required isotopy to the identity can be produced by applying the Alexander trick. Theorem~\ref{corollaryC} is a new mechanism for producing isotopies in manifolds with boundary, making it possible for us to follow the proof strategy above when the boundary is nonempty and the $4$-manifold is not $D^4$.  An initial example appeared in \cite[Theorem~F]{OP-mcgs}, where we observed that one can apply Theorem~\ref{corollaryC} to extend the results of \cite{CP} to give ambient isotopies, under an additional condition.

As one might expect, Theorem~\ref{theoremA} contrasts with the smooth case.  An example of this follows from the work of Hayden~\cite{Hayden}. Let $D_0$ and $D_1$ be the exotic $\Z$-slice discs from~\cite{Hayden}, with common boundary $K \subseteq S^3$. Let $S_0$ and $S_1$ be the simple spines in the knot trace $X_0(K)$ obtained from capping off $D_0$ and $D_1$ respectively with the core of the added 2-handle. By Theorem~\ref{theoremA}, the $2$-spheres  $S_0$ and $S_1$ are topologically isotopic. However, if they were smoothly isotopic, there would be a diffeomorphism of pairs $(X_0(K),S_0) \cong (X_0(K),S_1)$. The results of surgery on $S_0$ and $S_1$ would therefore be diffeomorphic.  But these surgeries result in $D^4 \sm \nu D_0$ and $D^4 \sm \nu D_1$ respectively, and the proof in~\cite{Hayden} showed that these 4-manifolds are in fact not diffeomorphic.

\subsection*{Organisation}

In Section~\ref{section:elements-modified-surgery} we recall the tools from Kreck's modified surgery that we will need.
In Section~\ref{section:properties-2-sphere-exterior} we begin to study simple 2-sphere spines of homotopy 2-spheres, by analysing the homotopy and spin types of their exteriors. As a consequence we determine the normal 2-type of the exterior.
In Section~\ref{sec:obstruction} we begin the modified surgery procedure, fixing two simple spines $S_0$ and $S_1$ and building 3-connected maps from their exteriors to the Postnikov $2$-type, that are moreover compatible with each other on the boundary. In Section~\ref{sec:homeo-between-exteriors} we apply modified surgery to show that $S_0$ and $S_1$ have homeomorphic exteriors.
We complete the proof of Theorem~\ref{theoremA} in Section~\ref{section:the-proof}.

\subsection*{Acknowledgments}
We thank Daniel Kasprowski for encouraging us to prove Lemma~\ref{lem:regularhomotopy}, Lisa Piccirillo for discussions on exotic spines, and Marco Golla for a potentially helpful suggestion which turned out not to be needed.
We also  thank the anonymous referee, whose insightful questions and comments led to improvements in the article.    We both gratefully acknowledge the MPIM, where we were visitors for part of the time this paper was written. MP warmly thanks ETH Z\"urich and Peter Feller for hospitality during a productive visit. PO was partially supported by the SNSF Grant~181199, and  MP was partially supported by EPSRC New Investigator grant EP/T028335/2 and EPSRC New Horizons grant EP/V04821X/2.

\subsection*{Corrigendum}
We thank Anthony Conway for pointing out a mistake in the proof of Theorem~\ref{theoremA} that appeared in the published version of this article. We had incorrectly quoted that the surgery obstruction in Theorem~\ref{thm:kreckmain} was valued in the Wall group $\displaystyle{L^s_5(\Z[\pi])}$, rather than in the correct group $\displaystyle{L^{s,\tau}_5(\Z[\pi])}$. In general these groups are not the same. However, when~$\Wh(\pi)=0$ these groups are the same (see the sequence~\eqref{eq:kreck} in Section~\ref{section:elements-modified-surgery}). Also when~$\Wh(\pi)$ is torsion-free and has trivial involution, these groups are the same (see Proposition~\ref{prop:fiddle}). This is sufficient to show that  $\displaystyle{L^{s,\tau}_5(\Z[\pi])} = \displaystyle{L^{s}_5(\Z[\pi])}$, for the cyclic groups $\pi$ with which we work.

%======================================================

\section{Elements of modified surgery theory}\label{section:elements-modified-surgery}

Many of the arguments we make in this article will use Kreck's modified surgery theory~\cite{MR1709301}. We now collect some definitions and results from this theory, for use later on. In this section, all manifolds are compact and oriented.

A \emph{cobordism rel.\ boundary} $(W,G_0,G_1)$ between $n$-dimensional manifolds with (possibly empty) boundary $X_0$ and $X_1$, with $\partial X_0 \cong Y \cong \partial X_1$, is an $(n+1)$-dimensional manifold $W$, with a decomposition of the boundary into codimension 0 submanifolds
\[
\partial W=\partial_0W\cup(Y\times[0,1])\cup-\partial _1W,
\]
together with homeomorphisms $G\colon (X_i,\partial X_i)\cong (\partial_iW,Y)$ for $i=0,1$.
In case that $Y=\emptyset$, a cobordism rel.~boundary is called a \emph{cobordism}.

Let $B$ be a space with the homotopy type of a CW complex. A map from a manifold $X\to B$ is called a \emph{$B$-structure}. Let $(Y,\zeta)$ be a closed $(n-1)$-dimensional manifold with $B$-structure $\zeta\colon Y\to B$. Given an $n$-manifold $X$, a \emph{$B$-structure rel.~boundary} $(g,\xi)$ (with respect to $(Y,\zeta)$) consists a homeomorphism $g\colon \partial X\cong Y$ and a map $\xi\colon X\to B$ such that $\zeta\circ g=\xi|_{\partial X}$. Given two $n$-manifolds $X_i$, with respective $B$-structures rel.~boundary $(g_i,\xi_i)$, a \emph{$B$-bordism rel.\ boundary} $(W,G_0,G_1,\Xi)$ is a cobordism rel.~boundary $(W,G_0,G_1)$ between $X_0$ and $X_1$, such that $G_i|_{\partial X_i}=g_i$, together with a $B$-structure $\Xi\colon W\to B$ such that $\Xi_i\circ G_i=\xi_i$ for $i=0,1$, and $\Xi|_{Y\times[0,1]}=\zeta\circ \pr_1$.

For an $n$-manifold $X$ we will write the stable topological normal bundle via its classifying map $\nu_X \colon X\to\BSTOP$ (see e.g.~\cite[\textsection 7]{TheGuide} for the definition). Let $(B,p)$ consist of a space $B$ with the homotopy type of a CW complex and $p\colon B\to \BSTOP$ a fibration. A map $\xi\colon X\to B$ is called a \emph{normal $B$-structure} on $X$ if the diagram
\[
\begin{tikzcd}
&& B\ar[d,"p"]\\
X\ar[rru,"\xi"]\ar[rr,"\nu_X"]&&\BSTOP
\end{tikzcd}
\]
commutes up to homotopy.

Let $(Y,\zeta)$ be a closed $(n-1)$-manifold $Y$ with normal $B$-structure. Given an $n$-manifold $X$, a \emph{normal $B$-structure rel.~boundary} $(g,\xi)$ (with respect to $(Y,\zeta)$) is a $B$-structure rel.\ boundary such that $\xi$ is moreover a normal $B$-structure, i.e.~$\xi$ is a lift of the stable normal bundle up to homotopy. Given two $n$-manifolds $X_i$, with respective $B$-structures rel.~boundary $(g_i,\xi_i)$, a \emph{normal $B$-bordism rel.\ boundary} $(W,G_0,G_1,\Xi)$ is a $B$-bordism rel.~boundary, such that $\Xi$ is moreover a normal $B$-structure. In the case that $Y=\emptyset$, we call $(W,G_0,G_1, \Xi)$ a \emph{normal $B$-bordism} and denote the corresponding bordism group of closed $n$-manifolds with normal $B$-structure by~$\Omega_n(B,p)$.

We record the following lemma for use later on; the proof is straightforward from the definitions and we omit it.

\begin{lemma}\label{lem:sanitycheck} Suppose $B$ is a space with the homotopy type of a CW complex and let $p\colon B\to\BSTOP$ be a fibration. Let $(Y,\zeta)$ be a closed $(n-1)$-manifold $Y$ with normal $B$-structure. For $i=0,1$, suppose that $X_i$ is an $n$-manifold with normal $B$-structure rel.\ boundary $(g_i,\xi_i)$. Define a closed $n$-manifold with normal $B$-structure
\[
(X,\xi):=(X_0\cup_{g_1^{-1}\circ g_0}-X_1, \xi_0\cup\xi_1).
\]
Then $(X,\xi)\sim0\in \Omega_n(B,p)$ is null-bordant if and only if $(X_0,g_0,\xi_0)$ and $(X_1, g_1, \xi_1)$ are normally $B$-bordant rel.~boundary.
\end{lemma}

A map of spaces $f\colon A\to B$ is \emph{$m$-connected} if $f_*\colon\pi_k(A)\to \pi_k(B)$ is an isomorphism for $k<m$ and is surjective for $k=m$. A map of spaces $f\colon A\to B$ is \emph{$m$-coconnected} if $f_*\colon\pi_k(A)\to \pi_k(B)$ is an isomorphism for $k>m$ and is injective for~$k=m$.

\begin{definition} A \emph{normal $2$-type} for a manifold $X$ is a pair $(B,p)$, where~$B$ is a space with the homotopy type of a CW complex, $p\colon B\to \BSTOP$ is a $3$-coconnected fibration with connected fibre, and $(B,p)$ is such that there exists a $3$-connected normal $B$-structure $\overline{\nu}_X\colon X\to B$. Such a normal $B$-structure $(X,\overline{\nu}_X)$ is called a \emph{normal $2$-smoothing} of $X$. If $(X,g,\overline{\nu}_X)$ is moreover a normal $B$-structure rel.\ boundary, with respect to some $(Y,\zeta)$, we say $(X,g,\overline{\nu}_X)$ is a \emph{normal $2$-smoothing rel.\ boundary}.
\end{definition}

Given a manifold $X$, the existence of a normal $2$-type follows from the theory of Moore-Postnikov decompositions (see~\cite{zbMATH03562120}). This theory furthermore guarantees that any two normal $2$-types for a given $X$ are fibre homotopy equivalent to one another.

\begin{definition}\label{def:hcobordism} An \emph{$h$-cobordism rel.~boundary} between $X_0$ and $X_1$ is a cobordism rel.~boundary $(W,G_0,G_1)$ such that each map $G_i\colon X_i\to W$ is a homotopy equivalence. An $h$-cobordism rel.~boundary is moreover an \emph{$s$-cobordism rel.~boundary} if each homotopy equivalence $G_i$ has vanishing Whitehead torsion.
\end{definition}

In classical surgery theory~\cite{Wall-surgery-book}, obstructions to doing surgery to improve certain classes of maps to
simple homotopy equivalences are situated in the \emph{surgery obstruction groups}. These are abelian groups $\displaystyle{L^s_n(\Z[\pi])}$, depending on the dimension $n$, and fundamental group $\pi$, of the manifold.
A variant of these obstruction groups, denoted $\displaystyle{L^{s,\tau}_n(\Z[\pi])}$, appears in Kreck's modified surgery theory. Kreck~\cite[Lemma~4.5]{kreckmonograph} shows there is an exact sequence
\begin{equation}\label{eq:kreck}
0\to \displaystyle{L^s_n(\Z[\pi])}\to \displaystyle{L^{s,\tau}_n(\Z[\pi])}\to \Wh(\pi).
\end{equation}
\begin{proposition}\label{prop:fiddle}
For any group $\pi$ such that $\Wh(\pi)$ is torsion-free and has trivial involution, the inclusion map $\displaystyle{L^s_n(\Z[\pi])}\to \displaystyle{L^{s,\tau}_n(\Z[\pi])}$ is an isomorphism.
\end{proposition}
\begin{proof}
In~\cite[Lemma~4.5]{kreckmonograph}, Kreck identifies the image of the map
$\displaystyle{L^{s,\tau}_n(\Z[\pi])}\to \Wh(\pi)$ from~\eqref{eq:kreck}
%cokernel of the inclusion map $\displaystyle{L^s_n(\Z[\pi])}\to \displaystyle{L^{s,\tau}_n(\Z[\pi])}$ as isomorphic to
as the subgroup $U\subseteq \Wh(\pi)$, generated by the stable classes of unitary matrices $A$ over~$\Z\pi$. When the involution $\overline{\phantom{x}}\colon\Wh(\pi)\to \Wh(\pi)$ is trivial, the unitary condition~$[A]+\overline{[A]}=0$ implies~$2[A]=0$, so this group is generated by $2$-torsion elements. But if moreover $\Wh(\pi)$ is torsion-free, we deduce $[A]=0$. Hence $U=0$ and~$\displaystyle{L^s_n(\Z[\pi])}\cong \displaystyle{L^{s,\tau}_n(\Z[\pi])}$ by  \eqref{eq:kreck}.
\end{proof}

The following is the main result of modified surgery theory, for the purposes of this paper.

\begin{theorem}[{\cite[Theorem 6.1(b)]{kreckmonograph}}]\label{thm:kreckmain}Suppose $X_0$ and $X_1$ are connected $4$-manifolds, each with boundary homeomorphic to $Y$, fundamental group $\pi$ and normal $2$-type $(B,p)$, where $B$ is homotopy equivalent to a CW complex with finite $2$-skeleton. Let $\zeta\colon Y\to B$ be a normal $B$-structure. Then a normal $B$-bordism rel.\ boundary $(Z,G_0,G_1,\Xi)$, between normal $2$-smoothings $(X_0,g_0,\xi_0)$ and $(X_1,g_1,\xi_1)$ rel.\ boundary determines a surgery obstruction $\theta(Z,\Xi)\in \displaystyle{L^{s,\tau}_5(\Z[\pi])}$. The obstruction $\theta(Z,\Xi)$ vanishes if and only if $(Z,G_0,G_1,\Xi)$ is normally $B$-bordant rel.~boundary to some $(Z',G_0',G_1',\Xi')$, where $Z'$ is an $s$-cobordism rel.~boundary.
\end{theorem}

It will be important to be able to compute the bordism groups $\Omega_4(B,p)$, to which we now turn.

The normal $2$-types we will need later on will be of the following general type. Let $P$ be a CW complex with finite 2-skeleton and let
\[B:= P\times \BTOPSpin.\]
Write the canonical principal fibration $\gamma \colon \BTOPSpin \to \BSTOP$ and let $p := \gamma \circ \pr_2 \colon B \to \BSTOP$ be projection followed by~$\gamma$.

In this case a normal $B$-structure on a 4-manifold $M$ consists of a map $M \to P$ and a spin structure on the stable normal bundle of $M$.  A spin structure on the stable normal bundle determines and is determined by a spin structure on the stable tangent bundle. As a consequence, we can identify the group $\Omega_4(B,p)$ with the usual spin bordism group $\Omega_4^{\TOPSpin}(P)$.
The latter group can be computed using an Atiyah-Hirzebruch spectral sequence:
\[
E_2^{r,s}\cong H_r(P;\Omega_s^{\TOPSpin}) \implies \Omega_{r+s}^{\TOPSpin}(P),
\]
where we recall that the coefficients in the range of interest are $\Omega_s^{\TOPSpin}\cong \Z, \Z/2, \Z/2, 0, \Z$ for $s=0,1,2,3,4$ respectively.

The next proposition, due to Teichner, describes certain differentials in this spectral sequence in terms of Steenrod squares.

\begin{proposition}[{\cite[Lemma]{MR1214960}}]\label{cor:themeat}~In the Atiyah-Hirzebruch spectral sequence described above, we have the following differentials.
\begin{enumerate}[leftmargin=*]\setlength\itemsep{0em}
\item\label{item:themeat1} The differential $d^{r,1}_2\colon H_r(P;\Z/2)\to H_{r-2}(P;\Z/2)$ is identified via universal coefficients with $\Hom(\Sq^2,\Z/2)$, where $\Sq^2$ is the second Steenrod square. More precisely, the following square commutes:
\[\begin{tikzcd}
    H_r(P;\Z/2) \ar[rr,"d_2^{r,1}"] \ar[d,"\cong"]
    && H_{r-2}(P;\Z/2) \ar[d,"\cong"] \\
  \Hom(H^r(P;\Z/2),\Z/2) \ar[rr,"(\Sq^2)^*"]
   && \Hom(H^{r-2}(P;\Z/2),\Z/2).
  \end{tikzcd}\]
\item\label{item:themeat2}
The differential $d_2^{r,0}\colon H_r(P;\Z)\to H_{r-2}(P;\Z/2)$ is reduction mod 2 followed by $\Hom(\Sq^2,\Z/2)$. More precisely, the following diagram commutes:
  \[\begin{tikzcd}
    H_r(P;\Z) \ar[rrrr,"d_2^{r,0}"] \ar[d,"\operatorname{red}_2"] &&&& H_{r-2}(P;\Z/2) \ar[d,"\cong"] \\
   H_r(P;\Z/2)\ar[rr,"\cong"] &&\Hom(H^r(P;\Z/2),\Z/2) \ar[rr, "(\Sq^2)^*"] && \Hom(H^{r-2}(P;\Z/2),\Z/2).
  \end{tikzcd}\]
\end{enumerate}
\end{proposition}

\section{Properties of simple spine exteriors and the normal 2-types}\label{section:properties-2-sphere-exterior}

Let $X$ be a compact, oriented $4$-manifold such that $X\simeq S^2$. Fix once and for all an orientation of $S^2$.
Suppose $S\subseteq X$ is the image of a locally flat embedding of $S^2$ that represents a generator of $\pi_2(X)$. Write $W:=X\sm \nu S$. Write $n\in\Z$ for the self-intersection of $S$ in $X$. The closed tubular neighbourhood $\ol{\nu} S\subseteq X$ homeomorphic to $D^2\wt{\times}_n S^2$, the $D^2$-bundle over $S^2$ with Euler number $n$. Hence $\partial W=\partial X\sqcup -L$ where $L:=\partial\ol{\nu} S$ is homeomorphic to the lens space $L(n,1)$ (our convention is that $L(0,1)=S^1\times S^2$). Assume in addition that $\pi_1(W)$ is abelian, i.e.\ that $S$ is simple.

The purpose of this section is to compute a normal $2$-type of $W$.

\begin{lemma}\label{lem:easy}
The manifold $X$ has connected, nonempty boundary.
\end{lemma}

\begin{proof}
Since $X$ is simply connected it is orientable, and so $H_4(X)=0$ implies the boundary must be nonempty. Next, $0=H_3(X)\cong H^1(X,\partial X)\cong \Z^{r-1}$, where $r$ is the number of boundary components. So $r=1$.
\end{proof}

\begin{lemma}\label{lem:lens}
The manifold $W$ is a $\Z$-homology bordism from $\partial X$ to $L$ and $\partial X\hookrightarrow W$ induces a surjection on $\pi_1$.
\end{lemma}

\begin{proof}
Consider the Mayer-Vietoris push-out square of singular chain complexes with $\Z$ coefficients
\[
\begin{tikzcd}
C_*(L)\ar[r]\ar[d]&C_*(W)\ar[d]\\
C_*(\overline{\nu}S)\ar[r]&C_*(X).
\end{tikzcd}
\]
In the homotopy category of chain complexes every push-out square is also a pull-back square. Since the square is a push-out and a pull-back, the fibre (resp.\ cofibre) of the top horizontal map is chain equivalent to the fibre (resp.\ cofibre) of the bottom horizontal map.
Since $S$ is a spine, the inclusion $\overline{\nu}S\hookrightarrow X$ is a homotopy equivalence, and so the lower horizontal arrow in the diagram is a chain equivalence. Therefore the fibre and cofibre of the bottom horizontal map are chain contractible. It follows that the same holds for the top map, so it is also a chain equivalence. Thus $L\to W$ is an integral homology equivalence. In particular, we have $H_*(W,L;\Z)=0$. By Poincar\'{e}-Lefschetz duality the cohomology is $H^*(W,\partial X;\Z)=0$. By the universal coefficient theorem the homology is $H_*(W,\partial X;\Z)=0$. Consequently $\partial X\to W$ is an integral homology equivalence, and so $W$ is indeed a $\Z$-homology bordism.

To see that $\partial X\hookrightarrow W$ induces a surjection on $\pi_1$, consider that as $\pi_1(W)$ is abelian, the map $\pi_1(\partial X)\to \pi_1(W)\cong H_1(W)$ factors through the surjection $\pi_1(\partial X)\to H_1(\partial X)$, so $\pi_1(\partial X) \to H_1(\partial X) \to H_1(W)$ is a composition of two surjections.
\end{proof}

\begin{lemma}\label{lem:lenshtpy}
The inclusion $L \hookrightarrow W$ induces    a homotopy equivalence  .
\end{lemma}

\begin{proof}
  The proof is via Whitehead's Theorem, so we must show the inclusion induced maps are isomorphisms on all homotopy groups. Since $\pi_1(W)$ is abelian, the Hurewicz theorem and Lemma~\ref{lem:lens} immediately imply the inclusion induces $\pi_1(L)\cong \pi_1(W)$. Now write $p\colon\widetilde{W}\to W$ and $p\colon\widetilde{L}\to L$ for the universal covers. Given the isomorphism on $\pi_1$, and the relative Hurewicz Theorem, it is now sufficient to show that inclusion induces isomorphisms $H_i(\widetilde{L})\cong H_i(\widetilde{W})$ for all $i\geq 2$~\cite[Theorem~IV.7.2]{MR516508}. As $\widetilde{L}$ is either $S^3$ (in the case $n\neq 0$) or $S^1\times S^2$ (in the case $n=0$), we have $H_i(\widetilde{L})=0$ for $i\geq 4$. Similarly, as $\widetilde{W}$ is a $4$-manifold with nonempty boundary, we have $H_i(\widetilde{W})=0$ for $i\geq 4$. Thus we focus our attention now on proving homology isomorphism for $i=2, 3$.

 Since, for each component of $\partial W$, the inclusion into $W$ induces a surjection on $\pi_1$, this implies that $\partial \wt{W}=p^{-1}(\partial W)$ has exactly two connected components. There is thus  an exact sequence
\[
\dots \to \underbrace{H_1(\wt{W};\Z)}_{=\,0} \to H_1(\wt{W},\partial \wt{W};\Z) \to \underbrace{H_0(\partial \wt{W};\Z)}_{\cong\,\Z^2}\twoheadrightarrow \underbrace{H_0(\wt{W};\Z)}_{\cong\,\Z}\to H_0(\wt{W},\partial \wt{W};\Z)\to0.
\]
The surjection $H_0(\partial \wt{W};\Z)\twoheadrightarrow H_0(\wt{W};\Z)$ splits, so we deduce that $H_1(\wt{W},\partial \wt{W};\Z)\cong \Z$ and $H_0(\wt{W},\partial \wt{W};\Z)=0$.

\underline{The case $n\neq 0$:} In this case $H_2(\widetilde{L})=0$ and $H_3(\widetilde{L})\cong\Z$ so we will be done if we show $H_2(\widetilde{W})=0$ and the inclusion induced map $H_3(\widetilde{L})\to H_3(\widetilde{W})$ is an isomorphism. As $n\neq 0$, we have that $\widetilde{W}$ is compact so, by Poincar\'{e}-Lefschetz duality and the universal coefficient theorem, we have
\[
H_3(\wt{W};\Z)\cong H^1(\wt{W},\partial \wt{W};\Z)
\cong H_1(\wt{W},\partial \wt{W};\Z)\cong \Z.
\]
Using Lemma~\ref{lem:lens} we compute that $\chi(W)=0$, and as Euler characteristic is multiplicative under finite covers, we have as well that $\chi(\wt{W})=0$. Setting $b_i=\dim_\Q(H_i(\wt{W};\Q))$ we calculate that
\[
0=b_0-b_1+b_2-b_3
=1-0+b_2-1
=b_2
\]
so that $H_2(\wt{W};\Z)$ is $\Z$-torsion. By the Universal Coefficient Theorem, the $\Z$-torsion of $H_2(\wt{W};\Z)$ would appear in $H^3(\wt{W};\Z)\cong H_1(\widetilde{W},\partial \widetilde{W};\Z)\cong\Z$. Thus $H_2(\wt{W};\Z)=0$ as required.

Now the long exact sequence of the pair gives an exact sequence
\[
H_4(\wt{W};\Z)\to H_4(\wt{W},\partial \wt{W};\Z)\to H_3(\wt{\partial X};\Z)\oplus H_3(\wt{L};\Z)\to H_3(\wt{W};\Z)\to H_3(\wt{W},\partial \wt{W};\Z)
\]
which is isomorphic to the sequence
\[
0\to\Z\to\Z\oplus \Z\to\Z\to 0,
\]
in which the second map is the diagonal $1\mapsto (1,1)$. We may conclude from this that the map $H_3(\wt{L};\Z)\to H_3(\wt{W};\Z)$ is an isomorphism as required (this is also true for $H_3(\wt{\partial X};\Z)\to H_3(\wt{W};\Z)$, but we do not need this).

\underline{The case $n=0$:} In this case, the Universal Coefficient Spectral Sequence for cohomology has $E_2$ page
\[
E_2^{r,s}=\Ext^s_{\Z[\Z]}(H_r(\wt{W},\partial\wt{W};\Z),\Z[\Z])
\]
and converges to $E_\infty^{r,s}=H^{r+s}(W,\partial W;\Z[\Z])$. The $r=0$ column vanishes on the $E_2$ page as $H_0(\wt{W},\partial\wt{W};\Z)=0$. For the $r=1$ column, recall that $H_1(\wt{W},\partial\wt{W};\Z)\cong\Z$. The standard cellular chain complex for the universal cover of $S^1\simeq B\Z$
gives a free $\Z[\Z]$-module resolution of~$\Z$, and hence we compute that $\Ext^s_{\Z[\Z]}(\Z;\Z[\Z]) \cong H^s(S^1;\Z[\Z]) \cong H_{1-s}(S^1;\Z[\Z])$. So $E_2^{1,s}\cong\Z$ for $s=1$, and vanishes for $s\neq 1$. Thus the only nonvanishing terms on the $r+s=2$ line are at $E^{1,1}_2$ and $E_2^{2,0}$, and as the bidegree of the $d_2$ differential is $(-1,2)$, the spectral sequence collapses here to yield a short exact sequence
$0 \to \Z = E^{1,1}_2 \to E^2_\infty= H^2(W,\partial W;\Z[\Z]) \to E^{2,0}_2 \to 0$.  Here $E^{2,0}_2 = \Hom_{\Z[\Z]}(H_2(W,\partial W;\Z[\Z]),\Z[\Z])$ is a free module by~\cite[Lemma 2.1]{BorodzikFriedlOnTheAlgebraic}. Therefore the short exact sequence splits and we have $H^2(W,\partial W;\Z[\Z]) \cong E_2^{2,0} \oplus \Z$.

We now claim that $H^2(W,\partial W;\Z[\Z])$ is $\Z[\Z]$-torsion. To see this, first consider $H_i(W;\Q(t))\cong H_i(W;\Z[\Z])\otimes_{\Z[\Z]}\Q(t)$, as localisation is flat. We have already seen that $E^{1,0}_2=E^{0,1}_2=0$, so we compute that $H^1(W,\partial W;\Z[\Z])=0$. Thus $H_3(W;\Z[\Z])=0$ by Poincar\'{e}-Lefschetz duality.
We also have $H_1(W;\Q(t)) \cong H_1(\wt{W};\Z) \otimes_{\Z[\Z]} \Q(t) = 0$ since $\pi_1(W)=\Z$, and $H_0(W;\Q(t)) \cong H_0(\wt{W};\Z) \otimes_{\Z[\Z]} \Q(t) \cong \Z \otimes_{\Z[\Z]} \Q(t) =0$.
This means we know $H_i(W;\Q(t))=0$ for $i\neq 2$. But as Euler characteristic can be computed with any field coefficients, this implies $\chi(W)=\dim_{\Q(t)}(H_2(W;\Q(t)))$. By Lemma~\ref{lem:lens} and the fact that $L$ is a closed, orientable 3-manifold, this number is 0, which proves the claim.

Combining the fact that $H^2(W,\partial W;\Z[\Z])\cong E^{1,1}_2\oplus E_2^{2,0}$ is $\Z[\Z]$-torsion with the fact that $E_2^{2,0}=\Ext^0_{\Z[\Z]}(H_2(\wt{W},\partial\wt{W};\Z),\Z[\Z])\cong\Hom_{\Z[\Z]}(H_2(\wt{W},\partial\wt{W};\Z),\Z[\Z])$ is free $\Z[\Z]$-module, we conclude that $E_2^{2,0}=0$. Thus $H^2(W,\partial W;\Z[\Z])\cong E^{1,1}_2\cong\Z$. From this we see
\[
\pi_2(W)\cong\pi_2(\wt{W})\cong H_2(\wt{W};\Z)\cong H^2(W,\partial W;\Z[\Z])\cong\Z
\]
as claimed.  Note that $\pi_2(L) \cong \pi_2(S^1 \times S^2) \cong \Z$ as well.

Write $i\colon L \to W$ for the inclusion. Consider the diagram
\[
\begin{tikzcd}
\pi_2(L)\ar[r, "\pi_2(i)"]\ar[d, "\text{Hur}"]&\pi_2(W)\ar[d, "\text{Hur}"]\\
H_2(L;\Z)\ar[r,"H_2(i)", "\cong"']&H_2(W;\Z)
\end{tikzcd}
\]
where the downwards maps are the Hurewicz maps. By the Hurewicz Theorem, these downwards maps are both surjective, and are thus both surjections from the infinite cyclic group to itself. Hence they are both isomorphisms. By the naturality of the Hurewicz map, the square commutes and so $\pi_2(i)$ is an isomorphism as claimed.

In other words, $H_2(\widetilde{L};\Z)\to H_2(\widetilde{W};\Z)$ is an isomorphism as required.

Along the way we have shown $H_3(\widetilde{W};\Z)=H_3(W;\Z[\Z])=0$, and as $L\cong S^1\times S^2$ we have~$H_3(\widetilde{L};\Z)=0$. Hence the proof is complete.
\end{proof}

We now have enough information to calculate a normal $2$-type of~$W$.

\begin{definition}\label{defn:normal-2-type-of-the-complement}
For $n\in\Z$, define
\[
B(n):=\left\{\begin{array}{rl}B(\Z/n)\times \BTOPSpin & \text{if $n\neq 0$,}\\ (S^1\times\C P^\infty)\times\BTOPSpin&\text{if $n=0$,}\end{array}\right.
\]
and define a fibration
\[
p_n\colon B(n)\xrightarrow{\gamma\circ\operatorname{pr}_2} \BSTOP,
\]
where $\operatorname{pr}_2$ is projection to the second factor and $\gamma\colon \BTOPSpin\to \BSTOP$ is the canonical map.
\end{definition}
\begin{proposition}\label{prop:normal2type_nonzero}
Let $n\neq 0$ be an integer. If $  M  $ is a compact, oriented, spin $4$-manifold with $\pi_1(  M  )\cong\Z/n$ and $\pi_2(  M  )=0$ then $  M  $ has normal $2$-type $(B(n),p_n)$.
\end{proposition}

\begin{proof}
First, $\pi_k(B(\Z/n))=0$ for $k\geq 2$ and $\gamma\colon \BTOPSpin\to \BSTOP$ is $2$-coconnected . Thus $p_n$ is $2$-coconnected, and is therefore $3$-coconnected as required.

Let $c\colon   M  \to B(\Z/n)$ denote the classifying map for the universal cover of $  M  $ and let $\mathfrak{s}\colon   M  \to \BTOPSpin$ denote a choice of spin structure. We claim $c\times\mathfrak{s}$ is a normal $2$-smoothing. Certainly $p_n \circ (c\times\mathfrak{s}) = \gamma \circ \mathfrak{s} = \nu  _M  $. As $\BTOPSpin$ is $3$-connected, the maps $\pi_k(  M  )\to \pi_k(B(\Z/n)\times\BTOPSpin)$ are clearly isomorphisms for $k=1,2$. We note as well that $\pi_3(B(\Z/n)\times\BTOPSpin)=0$, and it follows that $c\times \mathfrak{s}$ is $3$-connected as required.
\end{proof}

\begin{proposition}\label{prop:normal2type_zero}
If $X$ is a compact, oriented, spin $4$-manifold with $\pi_1(  M  )\cong\Z$ and $\pi_2(  M  )\cong \Z$, then $X$ has normal $2$-type $(B(0),p_0)$.
\end{proposition}

\begin{proof} First, $\pi_k(S^1\times\C P^\infty)=0$ for $k\geq 3$ and $\gamma\colon \BTOPSpin\to \BSTOP$ is $3$-coconnected. Thus $p_0$ is $3$-coconnected as required.

Since $H^3(S^1;\Z)=0$, the $k$-invariant of $  M  $ is trivial and therefore the Postnikov 2-type is a product $K(\Z,1)\times K(\Z,2) \simeq S^1 \times \C P^\infty$.
Let $c\colon   M  \to S^1\times\C P^\infty$ denote the $3$-connected map associated to the Postnikov 2-type of $  M  $ and let $\mathfrak{s}\colon   M  \to \BTOPSpin$ denote the choice of spin structure. We claim $c\times\mathfrak{s}$ is a normal $2$-smoothing. Certainly $p_0 \circ (c\times\mathfrak{s}) = \gamma \circ \mathfrak{s} = \nu  _M  $. As $\BTOPSpin$ is $3$-connected, the maps $\pi_k(  M  )\to \pi_k((S^1\times\C P^\infty)\times\BTOPSpin)$ are clearly isomorphisms for $k=1,2$. We note as well that $\pi_3((S^1\times\C P^\infty)\times\BTOPSpin)=0$, and it follows that $c\times\mathfrak{s}$ is $3$-connected.
\end{proof}

We have obtained the normal $2$-type of $W$.

\begin{proposition}\label{prop:2typeforexterior} Let $X$ be a compact, oriented $4$-manifold, and suppose that $X\simeq S^2$. Suppose $S\subseteq X$ is the image of a locally flat embedding of $S^2$ that represents a generator of $\pi_2(X)$. Write $W:=X\sm \nu S$, and write $n\in\Z$ for the normal Euler number of $S$ in $X$. Assume that $S$ is simple, i.e.\ that $\pi_1(W)$ is abelian. Then $(B(n),p_n)$ from Definition~\ref{defn:normal-2-type-of-the-complement} is a normal $2$-type of $W$.
\end{proposition}

\begin{proof}   By Lemma~\ref{lem:lenshtpy}, $W$ is homotopy equivalent to $L(n,1)$. As lens spaces are spin, and Steifel-Whitney classes of stable tangent bundles are homotopy invariants, we obtain that $W$ is spin.    When $n\neq 0$, Lemma~\ref{lem:lenshtpy} shows that $W$ satisfies the hypotheses of Proposition~\ref{prop:normal2type_nonzero}, and so the result follows. When $n=0$, Lemma~\ref{lem:lenshtpy} shows that $W$ satisfies the hypotheses of Proposition~\ref{prop:normal2type_zero}, and so the result follows.
\end{proof}

\section{Boundary-compatible maps to the Postnikov 2-type}\label{sec:obstruction}

As before, let $X$ be a compact, oriented 4-manifold with $X \simeq S^2$. Fix an orientation on $S^2$.
Suppose for $i=0,1$ that $S_i\subseteq X$ are simple spines, which are the images of maps representing the same generator of $\pi_2(X)$. Write $W_i:=X\sm \nu S_i$ and $L_i:=\partial\ol{\nu} S_i$. Write $n$ for the algebraic self-intersection of $S_0$ (and therefore also $S_1$) in $X$.

In this section, we begin the modified surgery programme for showing that $S_0$ and $S_1$ are ambiently isotopic. Write $P(n)$ for the Postnikov $2$-type of $W_i$; that is, the space such that~$B(n)=P(n)\times \BTOPSpin$.

We will prove that, given some map $\partial X\sqcup L(n,1)\to P(n)$, it is possible to extend it to a $3$-connected map $W_i\to P(n)$.
This is the first step to producing a full normal $2$-smoothing rel.~boundary for $W_i$, which we will need later. It follows from Proposition~\ref{prop:2typeforexterior} that the Postnikov $2$-type of $X$ is given by $P(n)=B(\Z/n)$ when $n\neq 0$ and $P(0)=B\Z\times \CP^\infty$, so when $n\neq 0$ a 3-connected map is the same as a $\pi_1$-isomorphism, and for $n=0$ it is the same as an isomorphism on $\pi_1$ and on $\pi_2$. To precisely phrase our main result of the section, we need a definition.

\begin{definition}\label{def:glueup}
Let $g_i\colon L_i\xrightarrow{\cong}L(n,1)$ be orientation preserving homeomorphisms. For $i=0,1$ we define
\[
W(g_0,g_1):=W_0\cup-W_1,\qquad \text{glued via}\qquad
\Id_{\partial X}\sqcup \,(g_1^{-1}\circ g_0)\colon \partial X\sqcup L_0\xrightarrow{\cong} \partial X\sqcup L_1.
\]
\end{definition}

Using this definition, the following lemma will achieve the aim described above.

Recall $L_i= \partial \overline{\nu}S_i$. It will be important in later sections of the paper that the parametrisations $g_i\colon \partial \overline{\nu}S_i\cong L(n,1)$ extend to    orientation preserving    homeomorphisms $G_i\colon \overline{\nu}S_i\cong D^2\wt{\times}_n S^2$. Hence we also build this into the lemma.

\begin{lemma}\label{lem:ellcompatible}
Let $n \in \Z$. For $i=0,1$, there are choices of    orientation preserving homeomorphisms     $G_i\colon \overline{\nu}S_i\cong D^2\wt{\times}_n S^2$, that preserve the 0-section, such that for the boundary parameterisations $g_i\colon L_i\cong L(n,1)$ obtained by restricting $G_i$, there exists a map $\ell\colon W(g_0,g_1)\to B(\Z/n)$ with the property that $\ell$ restricted to $W_i$ induces an isomorphism on $\pi_1$ for $i=0,1$.

In the case that $n=0$, the disc bundle structures, and hence the $g_0, g_1$, may be furthermore chosen such that there exists a map $\eta\colon W(g_0,g_1)\to \CP^\infty$ with the property that $\eta$ restricted to $W_i$ induces an isomorphism on $\pi_2$ for $i=0,1$.
\end{lemma}

\begin{remark}
  Note that for $n = \pm 1$ we have $L(\pm 1,1) \cong S^3$, and this lemma holds trivially. Nevertheless the proof goes through in this case without modification, so we shall not separate this case.  In the case $n=0$ recall that $L(0,1) \cong S^1 \times S^2$.
\end{remark}

To prove Lemma~\ref{lem:ellcompatible}, we will use the following technical result.

\begin{lemma}\label{lem:regularhomotopy}
For $i=0,1$ let $g_i\colon L_i\cong L(n,1)$ be any two choices of    orientation preserving    boundary parametrisation. Consider the diagram
\begin{equation}\label{eq:diagram}
\begin{tikzcd}[column sep=5em, row sep = 2em]
H_1(\partial X)\ar[r, "(j_0)_*", "\cong"']\ar[d, "\Id"'] & H_1(W_0)& H_1(L_0)\ar[r, "(g_0)_*", "\cong"']\ar[l, "(k_0)_*"', "\cong"] & H_1(L(n,1))\ar[d, "{\psi(g_0,g_1)}", "\cong"']\\
H_1(\partial X)\ar[r, "(j_1)_*", "\cong"']& H_1(W_1)& H_1(L_1)\ar[l, "(k_1)_*"', "\cong"]\ar[r, "(g_1)_*", "\cong"'] & H_1(L(n,1))
\end{tikzcd}
\end{equation}
where $\psi(g_0,g_1)$ is defined to make the diagram commute. Then the map $\psi(g_0,g_1)$ is multiplication by $\pm1$.
\end{lemma}

We defer the proof of this lemma until the very end of this section.
To take care of the sign ambiguity in $\psi(g_0,g_1)$ we will use the following self-homeomorphism of $L(n,1)$.

\begin{definition}\label{defn:tau}

Writing $D^2\times D^2\subseteq \C^2$, we may write $D^2\widetilde{\times}_n S^2$, the 2-disc bundle over $S^2$ with Euler number $n$, as $U_1\cup U_2$ where $U_i\cong D^2\times D^2$ and we glue together along $D^2\times S^1\subseteq U_i $ using $\phi(u,v)=(uv^{n},v)$. Note this restricts on the boundary of the total space of the disc bundle to a description of the lens space $L(n,1)$. Writing $S^1\times D^2\subseteq \C^2$, the lens space $L(n,1)$ is the identification space $V_1\cup V_2$ where $V_i\cong S^1\times D^2$ and we glue with the same formula along~$S^1\times S^1\subseteq V_i $

Define a homeomorphism $\tau\colon D^2\widetilde{\times}_n S^2\to D^2\widetilde{\times}_n S^2$ by the formula $\tau(u,v)=(\overline{u},\overline{v})$ when $(u,v)\in U_i$. Note that $\tau$ that restricts to a homeomorphism $\tau\colon L(n,1)\to L(n,1)$, using the same formula for $(u,v)\in V_i$.

\end{definition}

\begin{lemma}\label{lem:taucalculation}
For all $n$, the self-homeomorphism $\tau$ of $D^2\widetilde{\times}_n S^2$ is orientation preserving. Further, it induces multiplication by $-1$ on $H_2(D^2\widetilde{\times}_n S^2)\cong\Z$ and induces multiplication by $-1$ on $H_1(L(n,1))\cong\Z/n$.
\end{lemma}

\begin{proof}
The map $\tau$ is clearly orientation preserving as it comes from the composition of two complex conjugations.

A generator of $H_2(D^2\widetilde{\times}_n S^2)\cong \Z$ is given by the identification space $S^2\cong Y_1\cup Y_2\subseteq U_1\cup U_2$ where $Y_i:=\{0\}\times D^2\subseteq U_i$. Under~$\tau$, each hemisphere of this $2$-sphere is reflected by the complex conjugation $(0,v)\mapsto(0,\overline{v})$. Thus $\tau$ is orientation reversing on this $2$-sphere and the automorphism on $H_2(D^2\widetilde{\times}_n S^2)$ is multiplication by $-1$.

A generator of $H_1(L(n,1))$ is given by the oriented submanifold $S^1\times\{0\}\subseteq V_1$. Under~$\tau$, this circle is sent to itself by complex conjugation in the first factor, and so the orientation is switched. Thus the automorphism on $H_1(L(n,1))$ is multiplication by $-1$.
\end{proof}

\begin{proof}[Proof of Lemma~\ref{lem:ellcompatible} assuming Lemma~\ref{lem:regularhomotopy}] We choose orientation preserving homeomorphisms  $G_i\colon \ol{\nu} S_i\cong D^2\widetilde{\times}_n S^2$ that preserve the 0-section, with corresponding boundary parameterisations $g_i\colon  L_i\cong L(n,1)$. By Lemma~\ref{lem:regularhomotopy}, the map in Diagram (\ref{eq:diagram}) is $\psi(g_0,g_1)=\pm \Id$. By Lemma~\ref{lem:taucalculation} the self-homeomorphism $\tau$ of $D^2\widetilde{\times}_nS^2$ is such that the restriction to the boundary $L(n,1)$ induces multiplication by $-1$ on $H_1(L(n,1))$. By postcomposing $g_0$ with $\tau$, if necessary, we can and will assume that in fact $\psi(g_0,g_1)=\Id$.

In the case $n=0$, we make an additional observation. Note that $L(0,1)\cong S^1\times S^2$. Consider the version of Diagram~\eqref{eq:diagram} in second homology
\begin{equation}
\label{eq:diagram2}
\begin{tikzcd}[column sep=5em, row sep = 2em]
H_2(\partial X)\ar[r, "(j_0)_*", "\cong"']\ar[d, "\Id"'] & H_2(W_0)& H_2(L_0)\ar[r, "(g_0)_*", "\cong"']\ar[l, "(k_0)_*"', "\cong"] & H_2(S^1\times S^2)\ar[d, "{\psi'(g_0,g_1)}", "\cong"']\\
H_2(\partial X)\ar[r, "(j_1)_*", "\cong"']& H_2(W_1)& H_2(L_1)\ar[l, "(k_1)_*"', "\cong"]\ar[r, "(g_1)_*", "\cong"'] & H_2(S^1\times S^2)
\end{tikzcd}
\end{equation}
where $\psi'(g_0,g_1)$ is defined to make the diagram commute. As $H_2(S^1\times S^2)\cong\Z$, the map $\psi'(g_0,g_1)$ must be multiplication by $\pm1$. We claim that as $\psi(g_0,g_1)$ is the identity, so is $\psi'(g_0,g_1)$. To see this, let $x \in H_1(S^1\times S^2)$ and $y\in H_2(S^1\times S^2)$ be generators that intersect algebraically $+1$ in $S^1\times S^2$. As the $g_i$ are orientation preserving, $(g_i)_*^{-1}x$ and $(g_i)_*^{-1}y$ intersect algebraically $+1$ in $L_i$, for $i=0,1$. As $W_0$ is a homology bordism (Lemma~\ref{lem:lens}), by naturality of cup products the corresponding $(j_0)_*^{-1}\circ (k_0)_*\circ(g_0)_*^{-1}x \in H_1(\partial X)$ and $(j_0)_*^{-1}\circ (k_0)_*\circ(g_0)_*^{-1}y \in H_2(\partial X)$ intersect algebraically $+1$ in $\partial X$. Since $W_1$ is a homology bordism, their images under $(g_1)_* \circ (k_1)_*^{-1} \circ (j_1)_* \circ \Id$, in $H_1(S^1 \times S^2)$ and $H_2(S^1 \times S^2)$ respectively, also intersect algebraically~$+1$.
 By commutativity of diagrams \eqref{eq:diagram} and \eqref{eq:diagram2}, these images are the classes $\psi x=x$ and $\psi'y=\pm y$ respectively.
 %intersect algebraically $+1$ in $S^1\times S^2$.
 Thus $\psi' y=y$, and the claim is proven.

%OLD VERSION
%where $\psi'(g_0,g_1)$ is defined to make the diagram commute. As $H_2(S^1\times S^2)\cong\Z$, the map $\psi'(g_0,g_1)$ must be multiplication by $\pm1$. We claim that as $\psi(g_0,g_1)$ is the identity, so is $\psi'(g_0,g_1)$. To see this, let $x \in H_1(S^1\times S^2)$ and $y\in H_2(S^1\times S^2)$ be generators that intersect algebraically $+1$ in $S^1\times S^2$. As the $g_i$ are orientation preserving, $(g_i)_*^{-1}x$ and $(g_i)_*^{-1}y$ intersect algebraically $+1$ in $L_i$, for $i=0,1$. As $W_i$ is a homology bordism (Lemma~\ref{lem:lens}), by naturality of cup products the corresponding $(j_i)_*^{-1}\circ (k_i)_*\circ(g_i)_*^{-1}x \in H_1(\partial X)$ and $(j_i)_*^{-1}\circ (k_i)_*\circ(g_i)_*^{-1}y \in H_2(\partial X)$ intersect algebraically $+1$ in $\partial X$. By commutativity of diagrams \eqref{eq:diagram} and \eqref{eq:diagram2}, we deduce that the classes $\psi x=x$ and $\psi'y=\pm y$ intersect algebraically $+1$ in $S^1\times S^2$. Thus $\psi' y=y$, and the claim is proven.

Choose a map $\alpha\colon L(n,1)\to B(\Z/n)$ inducing an isomorphism $\pi_1(L(n,1))\cong\Z/n$. For $i=0,1$, choose a map $\beta_i\colon \partial X\to B(\Z/n)$ inducing the map
\[
\pi_1(\partial X)\xrightarrow{(j_i)_*} \pi_1(W_i)\xrightarrow{(k_i)_*^{-1}} \pi_1(L_i)\xrightarrow{(g_i)_*}\pi_1(L(n,1))\xrightarrow{\alpha_*}\Z/n.
\]
For each of $i=0,1$, we may now extend $\beta_i\sqcup (\alpha\circ g_i)\colon \partial X\sqcup L_i\to B(\Z/n)$ to a map $\ell_i\colon W_i\to B(\Z/n)$ inducing the isomorphism $\alpha_*\circ (g_i)_*\circ (k_i)_*^{-1}$ on $\pi_1$.

By commutativity of Diagram (\ref{eq:diagram}),    and the fact that $\psi(g_0,g_1)$ is the identity,    the restrictions $\ell_0|_{\partial W_0}$ and $\ell_1|_{\partial W_1}$ agree under the glueing map $\Id_{\partial X}\sqcup (g_1^{-1}\circ g_0)$, but only up to homotopy. Modify $\ell_0$ by a homotopy in a boundary collar of $\partial W_0$, to arrange that $\ell_0|_{\partial W_0}$ and $\ell_1|_{\partial W_1}$ agree under the glueing map. Now, together, $\ell_0$ and $\ell_1$ define a map $\ell\colon W(g_0,g_1)\to B(\Z/n)$ with the desired properties. When $n\neq 0$, this completes the proof of the lemma.

Now assume $n=0$ and consider the part of the lemma that remains to be shown. We must produce a map $\eta\colon W(g_0,g_1)\to \CP^\infty=K(\Z,2)$ so that its restriction to $W_i$ is an isomorphism on $\pi_2$, for $i=0,1$. But the method is entirely analogous to that in the previous two paragraphs, noting only that we have already shown the map $\psi'(g_0,g_1)$ of diagram \eqref{eq:diagram2} is the identity map. We omit further details.
\end{proof}

\begin{remark}
  Later, in Lemma~\ref{lem:CPinftycompatible}, we will modify the map $\eta$ just constructed. For now, it is merely important to know that one possible map $\eta$ exists.
\end{remark}

The remainder of this section is devoted to the proof of Lemma~\ref{lem:regularhomotopy}. We begin by considering how different choices of boundary parameterisations $g_0$ and $g_1$ affect the map $\psi$ in Diagram~\eqref{eq:diagram}.

Suppose $g_i\colon L_i\cong L(n,1)$ and $g'_i\colon L_i\cong L(n,1)$ are choices of homeomorphism, for $i=0,1$. Then $g'_i=(g'_ig_i^{-1})\circ g_i$. In other words, $g_i$ and $g'_i$ differ by a self-homeomorphism of $L(n,1)$. So to prove Lemma~\ref{lem:regularhomotopy}, it will be important for us to understand the possible automorphisms of $H_1(L(n,1))$ induced by self-homeomorphisms of $L(n,1)$.
Bonahon \cite{MR710104} computed the mapping class groups of lens spaces (see also \cite{MR823282}). We use this to determine the possible automorphisms of $H_1(L(n,1))$ induced by self-homeomorphisms of $L(n,1)$.

\begin{proposition}\label{prop:bonahon}
The map $g_*\colon H_1(L(n,1))\to H_1(L(n,1))$ induced by a homeomorphism $g\colon L(n,1)\cong L(n,1)$ is multiplication by $1$ or $-1$ on $\Z/n$ $($under our fixed identification $H_1(L(n,1))\cong\Z/n)$.
\end{proposition}

\begin{proof}
As $g_*$ is an automorphism of $\Z/n$, it is given by multiplication by some unit $\mu\in (\Z/n)^\times$. When $n=0$ the only units are $\mu=\pm1$, when $n=1$ the only unit is $\mu=1(=0)$, and when $n=2$ the only unit is $\mu=1$, so in these cases the statement is clear. For $n>2$, the mapping class group of $L(n,1)$ is generated by $\tau$ \cite[Th\'{e}or\`{e}me 3(c)]{MR710104} so $g$ is isotopic to either $\tau$ or the identity map,  and thus by Lemma~\ref{lem:taucalculation} the only possibilities are $\mu=\pm1$.
\end{proof}

\begin{corollary}\label{cor:signindependent}
Up to a sign, the homomorphism $\psi(g_0,g_1)\colon \Z/n\to\Z/n$ from Diagram~(\ref{eq:diagram}) is independent of the choices of $g_0$ and $g_1$.
\end{corollary}

To prove Lemma~\ref{lem:regularhomotopy}, we will show that there exists \emph{some} choice of parameterisations $g_0,g_1$ such that $\psi(g_0,g_1)=\pm1$. Then Corollary~\ref{cor:signindependent} shows any choice will return $\psi=\pm1$.

\begin{remark} We note that the map $\psi(g_0,g_1)$ must be multiplication by a unit in $\Z/n$, thus for any $n$ where $(\Z/n)^\times=\{1,-1\}$, the objective just outlined is trivially achieved. These are the cases $|n|=0,1,2,3,4,6$. However, the proof below is the same in all cases, so we proceed in generality.
\end{remark}

We build up some technology. Recall that we fixed an orientation of $S^2$. Let $f \colon S^2 \looparrowright X$ be a generic immersion that is also a homotopy equivalence. In a standard abuse of notation, henceforth we will conflate $f$ with the immersed submanifold given by its image $f(S^2)$.
  A closed regular neighbourhood $\overline{\nu}(f)$ is homeomorphic to the effect of performing some number of self-plumbings,~$k$ say, of the Euler number $n$ disc bundle over $S^2$. A standard Kirby diagram for this plumbed disc bundle is given in Figure~\ref{fig:plumbing} (see e.g.~\cite[p.~202]{MR1707327} or \cite[Diagram~2.2]{F}), where the clasps $C_i$ are one of the two possibilities depicted (it will not be relevant for us which ones). We identify the tubular neighbourhood with this standard description.

\begin{figure}
\begin{tikzpicture}[scale=0.95]

\begin{scope}[shift={(-0.4,0)}]
\begin{knot}[
%draft mode=crossings,
clip width=10, clip radius=5pt,
%consider self intersections ,
ignore endpoint intersections=false ,
]
\strand[thick] (7.8,2.05) ellipse (0.29cm and 0.65cm);
\filldraw[thick] (7.8,2.7) circle (2pt);
\strand[thick] (3.8,2.05) ellipse (0.29cm and 0.65cm);
\filldraw[thick] (3.8,2.7) circle (2pt);
\strand[thick] (0.8,2.05) ellipse (0.29cm and 0.65cm);
\filldraw[thick] (0.8,2.7) circle (2pt);
\strand[thick] (0,0) -- (9,0) to [out=right, in=down] (9.3,0.3) -- (9.3,0.7) to
% first bendover (right to left)
[out=up, in=right] (9, 1) -- (7.5,1) to [out=left, in=down] (7.2, 1.3) -- (7.2, 1.5) to [out=up, in=left] (7.5,1.8) -- (8, 1.8) to [out=right, in=up] (8.3, 1.5) -- (8.3,1) to [out=right, in=left] (8.8, 1) -- (8.8, 1.5) to [out=up, in=right] (8,2.3) -- (7.5, 2.3) to [out=left, in=up] (6.7, 1.5) -- (6.7, 1.3) to [out=down, in=right] (6.4, 1)
% end of bendover, to straight line
-- (6,1)
;
\strand[thick]
%  second bendover
(5.2, 1) -- (3.5,1) to [out=left, in=down] (3.2, 1.3) -- (3.2, 1.5) to [out=up, in=left] (3.5,1.8) -- (4, 1.8) to [out=right, in=up] (4.3, 1.5) -- (4.3,1) to [out=right, in=left] (4.8, 1) -- (4.8, 1.5) to [out=up, in=right] (4,2.3) -- (3.5, 2.3) to [out=left, in=up] (2.7, 1.5) -- (2.7, 1.3) to [out=down, in=right] (2.4, 1)
;
\strand[thick]
%  third bendover
(2.4, 1) -- (0.5,1) to [out=left, in=down] (0.2, 1.3) -- (0.2, 1.5) to [out=up, in=left] (0.5,1.8) -- (1, 1.8) to [out=right, in=up] (1.3, 1.5) -- (1.3,1) to [out=right, in=left] (1.8, 1) -- (1.8, 1.5) to [out=up, in=right] (1,2.3) -- (0.5, 2.3) to [out=left, in=up] (-0.3, 1.5) -- (-0.3, 1.3) to [out=down, in=right] (-0.6, 1)
;
\strand[thick]
(-0.6,1) -- (-0.7,1) to [out=left, in=up] (-1,0.7) -- (-1,0.3) to [out=down, in=left] (-0.7, 0) -- (0,0)
;
\flipcrossings{2,3,6,7,10,11}
\end{knot}

\draw[thick, fill=white] (8.2,1.35) rectangle (8.9,0.65);
\node at (8.56,1) {$C_k$};
\draw[thick, fill=white] (4.2,1.35) rectangle (4.9,0.65);
\node at (4.56,1) {$C_2$};
\draw[thick, fill=white] (1.2,1.35) rectangle (1.9,0.65);
\node at (1.56,1) {$C_1$};
\node at (5.6,1) {$\hdots$};
\node at (-1.3,0.4) {$n$};
\end{scope}

\begin{scope}[shift={(0.4,-0.7)}]
\node at (10.8,2.65) {$=$};
\draw[thick, fill=white] (11.1,3.2) rectangle (12.1,2.2);
\draw[thick, fill=white] (9.8,3) rectangle (10.5,2.3);
\node at (10.16,2.65) {$C_i$};
\node at (10.8,1.95) {or};
\begin{knot}[
%draft mode=crossings,
clip width=10, clip radius=5pt,
%consider self intersections ,
ignore endpoint intersections=false ,
]
\strand[thick] (11.1,2.7) -- (12.1, 2.7);
\strand[thick] (11.35,3.2) -- (11.35,2.55) to [out=down, in=left] (11.55, 2.35) -- (11.65,2.35) to [out=right, in=down] (11.85, 2.55) -- (11.85, 3.2);
\flipcrossings{1}
\end{knot}
\end{scope}

\begin{scope}[shift={(0.4,-2.2)}]
\node at (10.8,2.65) {$=$};
\draw[thick, fill=white] (11.1,3.2) rectangle (12.1,2.2);
\draw[thick, fill=white] (9.8,3) rectangle (10.5,2.3);
\node at (10.16,2.65) {$C_i$};
\begin{knot}[
%draft mode=crossings,
clip width=10, clip radius=5pt,
%consider self intersections ,
ignore endpoint intersections=false ,
]
\strand[thick] (11.1,2.7) -- (12.1, 2.7);
\strand[thick] (11.35,3.2) -- (11.35,2.55) to [out=down, in=left] (11.55, 2.35) -- (11.65,2.35) to [out=right, in=down] (11.85, 2.55) -- (11.85, 3.2);
\flipcrossings{2}
\end{knot}
\end{scope}
\end{tikzpicture}
\caption{A Kirby diagram for the $D^2$-bundle over $S^2$ with Euler number~$n$, and with~$k$ self-plumbings.}\label{fig:plumbing}
\end{figure}
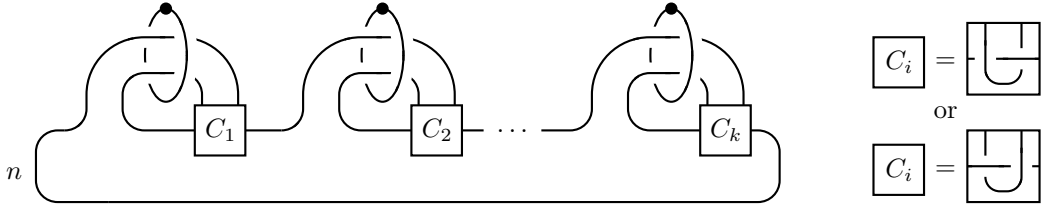

A loop $\gamma$ in $X$ is called a \emph{meridian} to $f$ if it is isotopic in $X\sm f$ to a meridian to the $2$-handle curve in the Kirby diagram.
%The tubular neighbourhood $\ol{\nu} (f)$ has the structure of a disc bundle, and thus $\partial\ol{\nu}(f)$ has the structure of an $S^1$-bundle. For any choice of disc bundle, and thence $S^1$-bundle structure on $\partial\ol{\nu}(f)$, we call the image of an $S^1$-fibre in the tubular neighbourhood a \emph{meridian} to $f$.
Note that as both $X$ and $S^2$ are oriented, there is a preferred orientation on a meridian to $f$. Namely, pick an embedded disc bounded by the meridian and intersecting the 2-handle geometrically once. Orient the disc so that this is a positive intersection and then restrict this orientation on the disc its boundary.

\begin{lemma}\label{lem:plumbing}
Assume that $f$ has $k$ double points. Then
\[
H_r(\overline{\nu}(f))\cong
\left\{\begin{array}{ll}
\Z & r=0,\\
\Z^k & r=1,\\
\Z & r=2,\\
0 & r\geq 3,
\end{array}
\right.
\qquad \qquad \qquad
H_r(\partial \overline{\nu}(f))\cong
\left\{\begin{array}{ll}
\Z & r=0,\\
\Z^k\oplus\Z/n & r=1,\\
\Z^k & r=2,\\
\Z & r=3,\\
0 & r\geq 4,
\end{array}
\right.
\]

The summand $\Z^k$ of both $H_1(\ol{\nu}(f))$ and $H_1(\partial\ol{\nu}(f))$ is generated by a collection of meridians to the dotted circles in Figure~\ref{fig:plumbing}.   The summand $\Z/n\subseteq H_1(\partial\ol{\nu}(f))$ is generated by a meridian of~$f$.
\end{lemma}

\begin{proof}From Figure~\ref{fig:plumbing} we can read off the claimed homology groups for $\ol{\nu}(f)$; it is homotopy equivalent to $S^2 \vee \bigvee^k S^1$.    One can deduce the homology groups of $\partial \overline{\nu}(f)$ from the long exact sequence of the pair $(\overline{\nu}(f),\partial \overline{\nu}(f))$, together with the fact that $H_2(\overline{\nu}(f))\to H_2(\overline{\nu}(f), \partial \overline{\nu}(f))$ can be identified with $\Z\xrightarrow{n}\Z$ by viewing it as the adjoint to the intersection pairing.
Alternatively, by switching the dots to 0's in Figure~\ref{fig:plumbing}, we obtain a link surgery diagram for $\partial \ol{\nu}(f)$, from which we can read off the claimed homology groups for this boundary manifold.
The claims about generators are clear from the picture.
\end{proof}

\begin{remark}
To each double point of $f$, we can assign a \emph{double point loop} on the image $f(S^2)$. This is a loop on the surface that leaves the double point on one sheet of the immersion and returns on the other, missing all other double points. The meridians to the dotted circles in Figure~\ref{fig:plumbing} are isotopic in $\ol{\nu}(f)$ to double point loops on the surface.
\end{remark}

\begin{proposition}\label{prop:immersionhomology}
 We have that that $H_1(X\sm f) \cong \Z/n$, generated by a meridian of $f$, and moreover the inclusion induced map $j_* \colon H_1(\partial X) \to H_1(X \sm f)$ is an isomorphism.
\end{proposition}

\begin{proof}
For the first claim, we observe that since $X$ is simply connected, there is an exact sequence
\[
\begin{tikzcd}
0\ar[r] & H_2(\partial X)\ar[r]& H_2(X)\ar[r, "\iota_*"]&H_2(X,\partial X)\ar[r]&H_1(\partial X)\ar[r]& 0.
\end{tikzcd}
\]
Using Poincar\'{e}-Lefschetz duality and the universal coefficient theorem, we identify $H_2(X,\partial X)\cong H_2(X)^*$. Under this identification $\iota_*$ becomes the adjoint of $\lambda_X$, and hence the sequence is isomorphic to
\[
\begin{tikzcd}
0\ar[r] & H_2(\partial X)\ar[r]& \Z\ar[r, "n"]&\Z\ar[r]&H_1(\partial X)\ar[r]& 0.
\end{tikzcd}
\]
From this, we deduce that $H_1(\partial X)\cong\Z/n$ (and $H_2(\partial X)=0$, although we will not need this).

We now move on to computing that $H_1(X\sm f)\cong \Z/n$. Write $W:=X\sm \overline{\nu}(f)$ and $Y:=\partial \overline{\nu}(f)$. Using the groups computed in Lemma~\ref{lem:plumbing}, the Mayer-Vietoris sequence for $X=W\cup \overline{\nu}(f)$ now shows that the inclusion  -induced    maps $H_3(Y)\to H_3(W)$, $H_2(Y)\to H_2(W)$ are isomorphisms, and that the    inclusion-induced map is an isomorphism
\begin{equation}
\label{eq:inclusion}
\begin{tikzcd}
\underbrace{H_1(Y)}_{\cong\,\Z^k\oplus\Z/n}\ar[r, "\cong"]& \underbrace{H_1(\ol{\nu}(f))}_{\cong\Z^k}\oplus H_1(W).
\end{tikzcd}
\end{equation}
From the classification of finitely generated abelian groups, we deduce that $H_1(W)\cong\Z/n$. The $\Z/n$ summand of $H_1(Y)$ is generated by a meridian of $f$, by Lemma~\ref{lem:plumbing}. The order of this meridian element is preserved under the map~\eqref{eq:inclusion}, as this is an isomorphism, in particular implying it maps to a generator of the torsion summand $H_1(W)$.   Moreover, taking these three facts above, combined with the long exact sequence of the pair $(W,Y)$, we may compute that $H_3(W,Y)=0$ and $H_2(W,Y)\cong\ker(H_1(Y)\to H_1(W))\cong \Z^k$.

For the final claim, consider that part of the long exact sequence for the pair $(W,\partial X)$ is
\[
\begin{tikzcd}
\underbrace{H_1(\partial X)}_{\cong\,\Z/n}\ar[r, "j_*"]& \underbrace{H_1(W)}_{\cong\,\Z/n}\ar[r]& H_1(W,\partial X).
\end{tikzcd}
\]
Using Poincar\'{e}-Lefschetz duality and the universal coefficient theorem, we identify $H_1(W,\partial X)\cong \Ext^1_\Z(H_2(W,Y),\Z)\oplus H_3(W,Y)^*$. But this group is 0 by the computation above. This shows $j_*$ is a surjective map $\Z/n\to\Z/n$, implying it is an isomorphism as required.
\end{proof}

We have shown that $j_*\colon H_1(\partial X)\xrightarrow{\cong} H_1(X\sm f)$ is an isomorphism, but we wish to be more careful about keeping track of which isomorphism it is. Fix a generator $\gamma \in H_1(\partial X) \cong \Z/n$. Let $k \colon \partial \ol{\nu}( f) \to X \sm \nu (f)$ be the inclusion map.

\begin{definition}
  A \emph{meridional marking} for $f$ is a homology class $\delta \in H_1(\partial \ol{\nu} (f))$ that contains a representative given by a meridian to $f$ with the preferred orientation. The meridionally marked immersion $(f,\delta)$ is said to be \emph{consistent} (with respect to $\gamma$) if $j_*(\gamma) = k_*(\delta) \in H_1(X \sm \nu (f)) \cong \Z/n$.
\end{definition}

We consider the behaviour of consistency under finger moves on the immersion.

\begin{proposition}\label{prop:reg-hom-invariance}
  Suppose that $f'\colon S^2\looparrowright X$ is obtained from $f$ by a self finger move. Let $\delta'\in H_1(\partial\ol{\nu} (f'))$ be a meridional marking for $f'$ obtained by taking a representative meridian for $\delta$ disjoint from a neighbourhood of the finger move arc. Then $(f',\delta')$ is consistent if and only if $(f,\delta)$ is consistent.
\end{proposition}

\begin{proof}
Suppose $(f,\delta)$ is consistent. Then there exists a surface $\Sigma\subseteq X\sm \nu(f)$ witnessing that $j_*(\gamma)$ and $k_*(\delta)$ are homologous. We may assume this surface is disjoint from a neighbourhood of the finger move arc. Write the inclusion $k'\colon \partial\ol{\nu} (f')\to X\sm \nu(f')$. Then $\Sigma$ witnesses that $j_*(\gamma)$ and $k'_*(\delta')$ are homologous in $X\sm\nu( f')$, so $(f',\delta')$ is consistent.

Conversely, let $\Sigma$ be a surface in $X\sm \nu(f')$ witnessing that $j_*(\gamma)$ and $k'_*(\delta')$ are homologous. Choose a Whitney disc $V$ reversing the finger move. We can assume that meridians in the markings $\delta$ and $\delta'$ are disjoint from the Whitney arcs. It may be that $V$ intersects $\Sigma$. If this is the case, perform finger moves on $\Sigma$, guided by arcs in $V$, to push intersections between $V$ and $\Sigma$ off $V$; see Figure~\ref{fig:pushingdown}, parts \subref{subfig:setup} and \subref{subfig:pushdown}. This is the procedure called \emph{pushing down}~\cite[\textsection2.5]{FQ}.
 \begin{figure}[h]
\subcaptionbox{A Whitney disc $V$ pairing self-intersections of $f$, together with a point of intersection between $\Sigma$ and $V$ (possible self-intersections of $V$ not pictured).\label{subfig:setup}}[0.3\linewidth]{
 \includegraphics{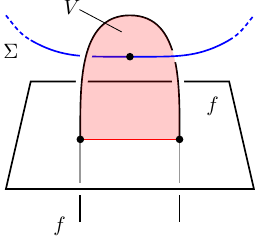}
}
\hfill
\subcaptionbox{{The surface $\Sigma$ after being pushed down. Also pictured, a choice of arc $A$ on $f$ between the new intersection points, and missing the Whitney arcs on $f$.\label{subfig:pushdown}}}[0.3\linewidth]{
\includegraphics{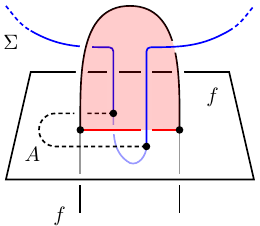}
}
\hfill
\subcaptionbox{{Tubing $\Sigma$ to itself, guided by the arc $A\subseteq f$.\label{subfig:tubing}}}[0.3\linewidth]{
\includegraphics{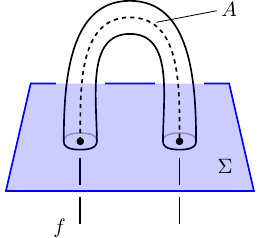}
}
\caption{Pushing intersections between $V$ and $\Sigma$ down, then tubing $\Sigma$ to itself.
}
 \label{fig:pushingdown}
\end{figure}
Each application of pushing down creates a pair of intersection points between~$\Sigma$ and~$f$. For each such pair, choose an arc on $f$ joining the intersections, which is disjoint from the Whitney arcs on $f$; see Figure~\ref{subfig:pushdown}. Using the normal directions to $f$, thicken the arc to a $3$-dimensional 1-handle, and then use the boundary of this $1$-handle to tube $\Sigma$ to itself; see Figure~\ref{subfig:tubing}. This gives a new $\Sigma$, disjoint from the Whitney disc, disjoint from $\nu(f)$, and witnessing a homology between $j_*(\gamma)$ and $k_*(\delta)$ in~$X\sm \nu(f)$. Therefore $(f,\delta)$ is consistent.
\end{proof}

We can finally prove that the map $\psi\colon \Z/n\to \Z/n$ from Diagram~(\ref{eq:diagram}) is always $\pm1$.

\begin{proof}[Proof of Lemma~\ref{lem:regularhomotopy}]
Write $f_0\colon S^2\hookrightarrow X$ and $f_1\colon S^2\hookrightarrow X$ for locally flat embeddings with images $S_0$ and $S_1$ respectively. As both $f_0$ and $f_1$ are embedded, homotopic, and have the same normal Euler number, they are regularly homotopic; see e.g.~\cite[Theorem~2.27]{KPRT}. A regular homotopy consists of a finite sequence of finger and Whitney moves.
Choose a meridional marking $\delta_0$ for $f_0$, represented by an oriented meridian $\mu_0\subseteq \partial\overline{\nu}(f_0)$ that is disjoint from neighbourhoods of all finger move arcs and Whitney discs in the finite sequence. The meridian $\mu_0$ thus survives the regular homotopy and becomes a meridian $\mu_1$ of $f_1$. We denote by $\delta_1$ the corresponding meridional marking for $f_1$. Choose a generator $\gamma\in H_1(\partial X)$ to arrange that $(f_0,\delta_0)$ is consistent. By Proposition~\ref{prop:reg-hom-invariance}, $(f_1,\delta_1)$ is also consistent.

Choose a fixed identification $H_1(L(n,1))\cong\Z/n$, so that $1\in\Z/n$ is represented by an $S^1$-fibre. For $r=0,1$, write $L_r:=\partial\ol{\nu}(f_r)$. By definition of a meridian, $\mu_r$ is the image of an $S^1$-fibre of $L(n,1)$ under some homeomorphisms $g_r \colon L_r\cong L(n,1)$, for $r=0,1$.
Under $(g_r)_*\colon H_1(L_r)\to  H_1(L(n,1))$, we have $(g_r)_*(\delta_r)=\pm1$, where $1\in H_1(L(n,1))$ refers to our fixed identification $H_1(L(n,1))\cong\Z/n$. The sign ambiguity comes from whether or not the homeomorphisms $g_0$ and $g_1$ agree with the orientation on the meridians determined by the orientation on $S^2$ and the ambient manifold $X$.
Consider Diagram~\eqref{eq:diagram}. As both $(f_0,\delta_0)$ and $(f_1,\delta_1)$ are consistent, we have $(j_0)_*(\gamma)=(k_0)_*(\delta_0)$ and $(j_1)_*(\gamma)=(k_1)_*(\delta_1)$. By the commutativity of \eqref{eq:diagram}, this implies $(g_1)_*^{-1}\circ \psi\circ (g_0)_*(\delta_0)=\delta_1$. But as $(g_r)_*(\delta_r)=\pm1$ for $r=0,1$, this implies $\psi(1)=\pm1$.

Since there exists a choice of $g_0, g_1$ such that $\psi(g_0,g_1)= \pm 1$, Corollary~\ref{cor:signindependent} now shows that for any choice, we must have $\psi=\pm1$.
\end{proof}

\section{Homeomorphisms between $2$-sphere exteriors}\label{sec:homeo-between-exteriors}

Recall $W_i:=X\sm \nu S_i$ and $L_i:=\partial\ol{\nu} S_i$. We now prove that there is a homeomorphism between the $2$-sphere exteriors $W_0$ and $W_1$.

\begin{proposition}\label{prop:exists2spherehomeo}
Let $X$ be a compact, oriented $4$-manifold and suppose that $X\simeq S^2$. Suppose for $i=0,1$ that $S_i\subseteq X$ are images of locally flat embeddings of $S^2$ that both represent a given generator of $\pi_2(X)$. Write $W_i:=X\sm \nu S_i$ and assume $\pi_1(W_i)$ is abelian. Then there is a homeomorphism $F\colon W_0\cong W_1$ that restricts to the identity map on $\partial X$ and on $\partial \ol{\nu} S_i \cong L(n,1)$, the latter for some choices of boundary parameterisations.
\end{proposition}

We will build up to the proof of Proposition~\ref{prop:exists2spherehomeo} with some lemmas. Recall, given choices of boundary parametrisations $g_i\colon L_i\xrightarrow{\cong}L(n,1)$, we defined $W(g_0,g_1)=W_0\cup-W_1$, glued using the $g_i$ on the $L_i$ boundary components and by the identity on $\partial X$; see Definition~\ref{def:glueup}.

\begin{lemma}\label{lem:spincompatible}
Under the hypotheses of Proposition~\ref{prop:exists2spherehomeo}, and for any two choices of $2$-disc bundle structure on the tubular neighbourhoods $G_i\colon \ol{\nu} S_i\cong D^2\widetilde{\times}_nS^2 $, that restrict to boundary parameterisations $g_i\colon L_i\cong L(n,1)$, there is a spin structure on $W(g_0,g_1)$.
\end{lemma}

\begin{proof}
When $n$ is odd, we compute $H^1(W_i;\Z/2)=0$, so that $W_i$ has a unique spin structure~$\mathfrak{s}_i$. Similarly, $\partial X$ and $L_i$ have unique spin structures because $W_i$ is a $\Z$-homology cobordism. Thus the restrictions of $\mathfrak{s}_0$ and $\mathfrak{s}_1$ to $\partial X$ agree, and so do the restrictions of $\mathfrak{s}_0$ and $\mathfrak{s}_1$ to $L_i$. This implies the chosen spin structures are compatible with the boundary glueing maps defining~$W(g_0,g_1)$.

When $n$ is even, recall that $X \simeq S^2$ is spin and has a unique spin structure. For $i=0,1$, endow $W_i$ with the spin structure $\mathfrak{s}_i$ restricted from the spin structure on $X$. On $L_i$, the restricted spin structure $\partial \mathfrak{s}_i$ is one of two possible spin structures on $L_i$. But exactly one of the spin structures on $L(n,1)$ extends to $D^2\widetilde{\times}_n S^2$. As the boundary parametrisation $L_i\cong L(n,1)$ extends to $\ol{\nu} S_i\cong D^2\widetilde{\times}_n S^2$, we see that for both $i=0,1$, the spin structure $\partial \mathfrak{s}_i\circ g_i^{-1}$ is this bounding spin structure on $L(n,1)$. Thus the spin structures $\partial \mathfrak{s}_0$ and $\partial\mathfrak{s}_1$ agree on $L_0$ and $L_1$ under the glueing map defining $W(g_0,g_1)$.
\end{proof}

\begin{lemma}\label{lem:CPinftycompatible}
Suppose the hypotheses of Proposition~\ref{prop:exists2spherehomeo}, and assume furthermore that $n=0$. Choose boundary parameterisations $g_i\colon L_i\cong S^1\times S^2$. Then $\sigma(W(g_0,g_1))=0$. Moreover, there exists a map $\eta\colon W(g_0,g_1)\to \CP^\infty$ so that the restriction $h_i:=\eta|_{W_i}$ is an isomorphism on $\pi_2$, and such that $\eta_*([W(g_0,g_1)])\in H_4(\CP^\infty;\Z)$ vanishes.
\end{lemma}

\begin{proof}
Fix choices of boundary parametrisation $g_i\colon L_i\cong S^1\times S^2$ for $i=0,1$. In this proof, write $W:=W(g_0,g_1)$, for brevity.

We compute the signature of $W$. As $n=0$, the original manifold $X$ has vanishing intersection form and thus signature 0. Similarly $\overline{\nu}S_i$ has signature~0. Since $X=W_i\cup\overline{\nu}S_i$, Novikov additivity now shows~$W_i$ has signature 0. A further application of Novikov additivity implies~$W$ has signature~0.

Now let $\eta\colon W(g_0,g_1)\to \CP^\infty$ be a map as produced by Lemma~\ref{lem:ellcompatible}.  We wish to describe how to modify $\eta$ so that $\eta_*([W])\in H_4(\CP^\infty;\Z)$ vanishes, but first we recast the condition in a more geometric way. As $H_4(\CP^\infty;\Z)$ is free, we have $\eta_*([W])=0$ if and only if $f(\eta_*([W]))=0$ for all homomorphisms $f\colon H_4(\CP^\infty;\Z)\to \Z$. By the universal coefficient theorem the evaluation map $H^4(\CP^\infty;\Z)\cong\Hom_{\Z}(H_4(\CP^\infty;\Z), \Z)$ is an isomorphism. So furthermore $\eta_*([W])=0$ if and only $\langle \psi, H_*([W])\rangle=0$ for all $\psi\in H^4(\CP^\infty;\Z)$. As there is an isomorphism of graded rings $H^*(\CP^\infty;\Z)\cong \Z[x]$ where $|x|=2$, this latter condition is equivalent to checking for the generator $x$ that $0=\langle x\cup x, \eta_*([W])\rangle=\langle \eta^*(x)\cup \eta^*(x), [W]\rangle$. So finally, $\eta_*([W])\in H_4(\CP^\infty;\Z)$ vanishes if and only if the Poincar\'{e} dual to $\eta^*(x)$ has vanishing self-intersection.

Now we describe the Poincar\'{e} dual to $\eta^*(x)$ and then show our choice of $\eta$ can be modified to a new map with the same properties but moreover so that this Poincar\'{e} dual has vanishing self-intersection.

We wish to obtain specific generators for $H_2(W)$. For this, consider the long exact sequence
\[
\begin{tikzcd}[row sep=0.5em]
\underbrace{H_2(\partial X)}_{\cong\Z}\oplus \underbrace{H_2(S^1\times S^2)}_{\cong\Z}\ar[r, "\alpha"] &\underbrace{H_2(W_0)}_{\cong\Z}\oplus \underbrace{H_2(W_1)}_{\cong\Z}\ar[r, "\beta"]&\phantom{.} \\
H_2(W)\ar[r, "\gamma"] & \underbrace{H_1(\partial X)}_{\cong\Z}\oplus \underbrace{H_1(S^1\times S^2)}_{\cong\Z}\ar[r, "\delta"]&\underbrace{H_1(W_0)}_{\cong\Z}\oplus \underbrace{H_1(W_1)}_{\cong\Z}
\end{tikzcd}
\]
where $\alpha$ and $\delta$ are both $\left(\begin{smallmatrix}1&1\\1&1\end{smallmatrix}\right)$. Write $L=\coker(\alpha)$ and $J=\ker(\delta)$, which we note are both free of rank 1. There is then an exact sequence
\begin{equation}\label{eq:reminder}
\begin{tikzcd}
0\ar[r] & L\ar[r, "\beta"] & H_2(W)\ar[r, "\gamma"]& J\ar[r] &0.
\end{tikzcd}
\end{equation}
From this we deduce that $H_2(W)\cong\Z\oplus \Z$. Now $L$ is generated by the class of a generator of $H_2(W_0)$, which is the image of a generator of $H_2(S^1\times S^2)$, say.
%As the generator of $L$ comes from a class in the boundary it has self-intersection 0. This implies $\beta(L)$ is self-annihilating and half-rank in $(H_2(W),\lambda_W)$. In other words $\beta(L)$ is a metaboliser. Since $W$ is a closed 4-manifold, the intersection form is nonsingular. Metabolic, nonsingular forms have vanishing signature, and so it follows that $\sigma(W)=0$.

Choose a loop $\lambda \subseteq \partial X$ generating $H_1(\partial X;\Z)$. For each of $i=0,1$ there exists a properly embedded annulus $A_i\subseteq W_i$ with $\lambda=A_i\cap \partial X$ and $g_i^{-1}(S^1\times\{\pt\})=A_i\cap L_i$. Fix choices of $A_i$. Glueing these annuli together, we obtain an embedded torus $T\subseteq W$ such that $\gamma([T])$ is a generator for $J$ in sequence (\ref{eq:reminder}). Write $S\subseteq W$ for the $2$-sphere $\{\pt\}\times S^2\subseteq S^1\times S^2\subseteq W$. Then the intersection form for $W$ is
\[
(H_2(W;\Z),\lambda_{X,\Z})=\left(\Z\langle [S], [T]\rangle\,\,,\, \begin{pmatrix} 0 & 1 \\ 1 & a\end{pmatrix}\right),
\]
where $a\in 2\Z$ is some unknown integer (which is necessarily even because $W$ is spin).

Consider for $i=0,1$ the commutative diagram
\[
\begin{tikzcd}
{[W,\CP^\infty]}\ar[r,"\cong"]\ar[d]
& H^2(W;\Z) \ar[r,"\cong"]\ar[d]
& \Hom_\Z(H_2(W;\Z),\Z)\ar[d]
\\
{[W_i,\CP^\infty]}\ar[r,"\cong"]
& H^2(W_i;\Z) \ar[r,"\cong"]
& \Hom_\Z(H_2(W_i;\Z),\Z)
\end{tikzcd}
\]
The Poincar\'{e} dual to $[S]\in H_2(W;\Z)$ is some class $y\in H^2(W;\Z)$, and as an element of $\Hom_\Z(H_2(W;\Z);\Z)$ is given by the left column of the intersection form $\lambda_{X,\Z}$. In particular, this map evaluates to 0 on $[S]$. As $[S]\in H_2(W_i;\Z)$ generates, and the diagram commutes, we see that $y$ maps to 0 under the vertical morphism. Write $u[S]+v[T]$ for the Poincar\'{e} dual class to $\eta^*(x)$. Let $b := a/2$ and let $G\colon W\to \CP^\infty$ be such that $G^*(x)=\eta^*(x)-(u+ bv)y$. Because~$y$ maps vertically to 0, the restrictions of $G$ and $\eta$ to $W_i$ are homotopic, for each $i=0,1$. In particular, $G$ restricted to $W_i$ is a $\pi_2$-isomorphism. The Poincar\'{e} dual to $G^*(x)$ is $u[S]+v[T]-(u+bv)[S]=-bv[S]+v[T]$, which has self-intersection:
\[\begin{pmatrix} -bv & v
\end{pmatrix}\begin{pmatrix} 0 & 1 \\ 1 & a
\end{pmatrix}\begin{pmatrix} -bv \\ v
\end{pmatrix} = -bv^2 - bv^2 + av^2 =0. \]
Thus $G$ is the replacement for $\eta$ we seek.
\end{proof}

Now we complete the proof of Proposition~\ref{prop:exists2spherehomeo}. The proofs in the cases that $n\neq 0$ and $n=0$ are structurally similar, but differ in the details. We will thus split the proof into these two cases for ease of digestion.

\begin{proof}[Proof of Proposition~\ref{prop:exists2spherehomeo}, assuming $n\neq 0$] By Proposition~\ref{prop:2typeforexterior}, the normal $2$-type of $W_i$ is $B(n)=\BTOPSpin\times B(\Z/n)\to \BSTOP$. A normal $2$-smoothing for $W_i$ consists of a choice of spin structure $\mathfrak{s}_i\colon W_i\to \BTOPSpin$, together with a choice of map $\ell_i\colon W_i\to B(\Z/n)$ that induces an isomorphism on $\pi_1$. By Lemma~\ref{lem:ellcompatible} and Lemma~\ref{lem:spincompatible} there exist choices of boundary parametrisation $g_i\colon L_i\cong L(n,1)$ and normal $2$-smoothings $\overline{\nu}_{W_i}=\mathfrak{s}_i\times \ell_i$ such that the normal $2$-smoothing
\[
\overline{\nu}_{W_0}\sqcup \overline{\nu}_{W_1}\colon W_0\sqcup -W_1\to B(n)= \BTOPSpin\times B(\Z/n)
\]
descends to a normal $B(n)$-structure on the closed $4$-manifold $W(g_0,g_1)$, which we will denote $\overline{\nu}_{W_0}\cup \overline{\nu}_{W_1}$. For the remainder of the proof, we abbreviate to $W:=W(g_0,g_1)$.

We wish to show that the normal $B(n)$-structure $(W, \overline{\nu}_{W_0}\cup \overline{\nu}_{W_1})$ vanishes in the group $\Omega_4(B(n),p_n)$. To prove this claim, we first compute the bordism group. Because $B(n)$ is a product, we have $\Omega_4(B(n),p_n)\cong\Omega^{\TOPSpin}_4(B(\Z/n))$. We claim that
\[
\Omega_4^{\TOPSpin}(B(\Z/n)) \cong \Omega_4^{\TOPSpin}\cong \Z.
\]
The second isomorphism is well known, and given by the signature of the $4$-manifold divided by~8, so we focus on the first. We use the Atiyah-Hirzebruch spectral sequence, which has $E_2^{r,s}\cong H_r(B(\Z/n);\Omega_s^{\TOPSpin})$. Recall that $\Omega_s^{\TOPSpin}\cong \Z,\Z/2,\Z/2,0,\Z$ for $s=0,1,2,3,4$ ($s$-manifolds for $s\leq 3$ admit unique smooth structures, so here we may use the smooth spin bordism groups, and for $s=4$, the reader is referred to~\cite[\textsection 13]{MR0423356}). A cellular chain complex for $B(\Z/n)$ is given by
\[
\dots\to\Z\xrightarrow{0}\Z\xrightarrow{n}\Z\xrightarrow{0}\Z\xrightarrow{n}\Z\xrightarrow{0}\Z,\]
and from this we can fill in some relevant terms of the $E^2$ page, focussing on when $r+s=4,5$. First note that $E_2^{0,4}\cong\Omega_4^{\TOPSpin}\cong \Z$. For $n$ odd, this is the only nonvanishing upper right quadrant entry on the diagonal line $r+s=4$.
All the  differentials $d_m$, for $m \geq 2$, with codomain $E_2^{0,4}$, have finite groups in their domains. Thus this terms survives to the $E_{\infty}$ page and so $\Omega_4^{\TOPSpin}(B(\Z/n)) \cong \Omega_4^{\TOPSpin}\cong \Z$ for $n$ odd.
For $n$ even, we can fill in relevant entries in the $E_2$ page as follows.

\begin{center}
\begin{tikzpicture}

\draw[step=1.0,black,xshift=1cm,yshift=1cm] (0,0) grid (6.5,5.5);

\node at (0.5,1.5) {$0$};
\node at (0.5,2.5) {$1$};
\node at (0.5,3.5) {$2$};
\node at (0.5,4.5) {$3$};
\node at (0.5,5.5) {$4$};
\node at (0.5,6.5) {$q$};

\node at (1.5,0.5) {$0$};
\node at (2.5,0.5) {$1$};
\node at (3.5,0.5) {$2$};
\node at (4.5,0.5) {$3$};
\node at (5.5,0.5) {$4$};
\node at (6.5,0.5) {$5$};
\node at (7.5,0.5) {$p$};

\node at (5.5,1.5) {$0$};
\node at (6.5,1.5) {$\Z/n$};

\node at (5.5,2.5) {$\Z/2$};
\node at (4.5,2.5) {$\Z/2$};

\node at (3.5,3.5) {$\Z/2$};
\node at (4.5,3.5) {$\Z/2$};

\node at (6.5,4.5) {$0$};
\node at (5.5,4.5) {$0$};
\node at (4.5,4.5) {$0$};
\node at (3.5,4.5) {$0$};
\node at (2.5,4.5) {$0$};
\node at (1.5,4.5) {$0$};

\node at (1.5,5.5) {$\Z$};

\node at (11,5) {$E_2^{r,s}\cong H_r(B(\Z/n);\Omega_s^{\TOPSpin})$};

\end{tikzpicture}
\end{center}

The differentials $d_2 \colon E_2^{5,0} \to E_2^{3,1}$ and $d_2 \colon E_2^{4,1} \to E_2^{2,2}$ are given by Proposition~\ref{cor:themeat}. Moreover, the former is known to be the surjective map $\Z/n \to \Z/2$, and the latter is known to be an isomorphism $\Z/2 \to \Z/2$ \cite[\textsection 6.2.1]{KPT}.  So on the $E_3$ page, the only nonzero term when $r+s=4$ is $E_3^{0,4}$, which is isomorphic to $\Z$ since $E_2^{2,3}=0$. Moreover, all the  differentials $d_m$, for $m \geq 3$, with codomain $E_m^{0,4}$, have finite groups in their domains, and therefore inductively we see that they map trivially to $E_r^{0,4} \cong \Z$. Thus this term survives to the $E_\infty$ page, and we have  $\Omega_4^{\TOPSpin}(B(\Z/n)) \cong \Omega_4^{\TOPSpin}\cong \Z$ for $n$ even as well.

To compute the signature of $W$ consider the section of Mayer-Vietoris sequence
\[
\underbrace{H_2(W_0)}_{=0}\oplus \underbrace{H_2(W_1)}_{=0}\to H_2(W)\to \underbrace{H_1(\partial X)}_{\cong\Z/n}\oplus \underbrace{H_1(L(n,1))}_{\cong\Z/n}
\]
implies that $H_2(W)$ is $\Z$-torsion. Hence $\sigma(W)=0$ and so $(W, \overline{\nu}_{W_0}\cup \overline{\nu}_{W_1})$ vanishes in the group $\Omega_4(B(n),p_n)$. By Lemma~\ref{lem:sanitycheck}, this implies that $(W_0,\overline{\nu}_{W_0})$ is cobordant to $(W_1,\overline{\nu}_{W_1})$ rel.~boundary. By Theorem~\ref{thm:kreckmain}, such a bordism determines a surgery obstruction in
$\displaystyle{L^{s,\tau}_5(\Z[\Z/n])}$, which vanishes only if $(W_0,\overline{\nu}_{W_0})$ is $s$-cobordant to~$(W_1,\overline{\nu}_{W_1})$ rel.~boundary.

We claim that $\displaystyle{L^{s,\tau}_5(\Z[\Z/n])}=0$. For this, by \cite[Theorem~7]{MR470029} %~\cite[Introduction]{MR1747537}
 the simple Wall group is trivial $\displaystyle{L_5^{s}(\Z[\Z/n])}=0$. Hence it is sufficient to show $\displaystyle{L^{s,\tau}_5(\Z[\Z/n])}\cong \displaystyle{L^{s}_5(\Z[\Z/n])}$, which will follow if the hypotheses of Proposition~\ref{prop:fiddle} are satisfied. These hypotheses hold: firstly, Oliver~\cite[Theorem 5.6]{MR933091} showed that $\Wh(\Z/n)$ is torsion-free; secondly, the involution on~$\Wh(\pi)$ is trivial~\cite[Proposition~4.2]{MR347999} (see also~\cite{MR451249}) for any finite abelian group, so in particular for $\pi\cong \Z/n$. This completes the proof of the claim.

As the value group for the surgery obstruction vanishes $\displaystyle{L^{s,\tau}_5(\Z[\Z/n])}=0$, the desired $s$-cobordism exists.
As $\Z/n$ is a good group, the Freedman-Quinn $s$-cobordism theorem~\cite[Theorem~7.1A]{FQ} implies there exists a homeomorphism $F\colon W_0\cong W_1$ restricting to the identity on $\partial X$ and to $g_1^{-1}\circ g_0\colon L_0\to L_1$ on this boundary component.
\end{proof}

\begin{proof}[Proof of Proposition~\ref{prop:exists2spherehomeo}, assuming $n= 0$]
By Proposition~\ref{prop:2typeforexterior}, the normal $2$-type of $W_i$ is $B(0)=\BTOPSpin\times S^1\times \CP^\infty\to \BSTOP$. A normal $2$-smoothing for $W_i$ consists of a choice of spin structure $\mathfrak{s}_i\colon W_i\to \BTOPSpin$, together with a choice of map $\ell_i\colon W_i\to S^1$ that induces an isomorphism on $\pi_1$, and a choice of map $h_i\colon W_i\to \CP^\infty$ that induces an isomorphism on $\pi_2$. By Lemmas~\ref{lem:ellcompatible} and~\ref{lem:spincompatible},  there exist choices of boundary parametrisation $g_i\colon L_i\cong S^1\times S^2$ and normal $2$-smoothings $\overline{\nu}_{W_i}=\mathfrak{s}_i\times \ell_i\times h_i$ such that the normal $2$-smoothing
\[
\overline{\nu}_{W_0}\sqcup \overline{\nu}_{W_1}\colon W_0\sqcup -W_1\to \BTOPSpin\times S^1\times \CP^\infty
\]
descends to a normal $B(0)$-structure on the closed $4$-manifold $W(g_0,g_1)$, denoted
\[
\mathfrak{s}\times\lambda\times\eta\colon W(g_0,g_1)\to \BTOPSpin\times S^1\times \CP^\infty.
\]
For the remainder of the proof, we abbreviate to $W:=W(g_0,g_1)$. By Lemma~\ref{lem:CPinftycompatible}, we may  choose the $h_i$ to be such that $\eta_*([W]) = 0 \in H_4(\CP^{\infty};\Z)$.

We wish to show that the normal $B(0)$-structure $(W, \mathfrak{s}\times\lambda\times\eta)$ vanishes in the group $\Omega_4(B(0),p_0)$. To prove this claim, we first compute the bordism group. Because $B(0)$ is a product, we have $\Omega_4(B(0),p_0)\cong\Omega^{\TOPSpin}_4(S^1\times\CP^\infty)$. We claim that
\[
\Omega^{\TOPSpin}_4(S^1\times\CP^\infty) \cong \Z\oplus\Z.
\]
where the first summand is given by the signature of $W$ divided by 8 and the second is given by $\eta_*([W])\in H_4(\CP^\infty;\Z)\cong\Z$ divided by 2. We fill in relevant entries in the $E_2$ page of the Atiyah-Hirzebruch spectral sequence as follows.

\begin{center}
\begin{tikzpicture}

\draw[step=1.0,black,xshift=1cm,yshift=1cm] (0,0) grid (6.5,5.5);

\node at (0.5,1.5) {$0$};
\node at (0.5,2.5) {$1$};
\node at (0.5,3.5) {$2$};
\node at (0.5,4.5) {$3$};
\node at (0.5,5.5) {$4$};
\node at (0.5,6.5) {$q$};

\node at (1.5,0.5) {$0$};
\node at (2.5,0.5) {$1$};
\node at (3.5,0.5) {$2$};
\node at (4.5,0.5) {$3$};
\node at (5.5,0.5) {$4$};
\node at (6.5,0.5) {$5$};
\node at (7.5,0.5) {$p$};

\node at (5.5,1.5) {$\Z$};
\node at (6.5,1.5) {$\Z$};

\node at (5.5,2.5) {$\Z/2$};
\node at (4.5,2.5) {$\Z/2$};
\node at (3.5,2.5) {$\Z/2$};
\node at (2.5,2.5) {$\Z/2$};
\node at (1.5,2.5) {$\Z/2$};

\node at (3.5,3.5) {$\Z/2$};
\node at (2.5,3.5) {$\Z/2$};

\node at (6.5,4.5) {$0$};
\node at (5.5,4.5) {$0$};
\node at (4.5,4.5) {$0$};
\node at (3.5,4.5) {$0$};
\node at (2.5,4.5) {$0$};
\node at (1.5,4.5) {$0$};

\node at (1.5,5.5) {$\Z$};

\node at (11,5) {$E_2^{r,s}\cong H_r(S^1\times \CP^\infty;\Omega_s^{\TOPSpin})$};

\end{tikzpicture}
\end{center}

The $d_2$ differentials are given by Proposition~\ref{cor:themeat}, and we compute them now. For $R=\Z/2$ or $R=\Z$, the K\"{u}nneth theorem gives
\[
H^*(S^1\times \CP^\infty;R)\cong H^*(S^1;R)\otimes H^*(\CP^\infty;R)\cong (R[t]/t^2)\otimes R[x]
\]
where $t\in H^1(S^1;R)$ and $x\in H^2(\CP^\infty;R)$ are generators. We compute that $\Sq^2(t\otimes 1)=0$, $\Sq^2(1\otimes x)=1\otimes x^2$, and by the Cartan formula
\begin{align*}
\Sq^2(t\otimes x)&=\Sq^2(t)\otimes x+\Sq^1(t)\otimes \Sq^1(x)+t\otimes\Sq^2(x)\\
&=0\otimes x+0\otimes 0+t\otimes x^2=t\otimes x^2.
\end{align*}
Thus $d_2^{3,1}=0$, $d_2^{4,1}$ is an isomorphism, and both $d_2^{4,0}$, $d_2^{5,0}$ are surjective with kernel $2\Z$. From this we compute $E_3^{2,2}=E_3^{3,1}=0$, and $E_3^{4,0} \cong 2\Z$.  There is potentially a nontrivial differential $d_3^{4,0} \colon E_3^{4,0} \cong 2\Z \to E_3^{1,2} \cong \Z/2$, and $E_{\infty}^{4,0} \cong \ker d_3^{4,0}$.

A choice of map $\pt \to S^1\times\CP^\infty$ induces $\Omega^{\TOPSpin}_4 \to \Omega^{\TOPSpin}_4(S^1\times\CP^\infty)$, which splits the map $\Omega^{\TOPSpin}_4(S^1\times\CP^\infty) \to \Omega^{\TOPSpin}_4$ induced by the unique map $S^1\times\CP^\infty \to \pt$. It follows that $\Omega^{\TOPSpin}_4$ is a direct summand of $\Omega^{\TOPSpin}_4(S^1\times\CP^\infty)$, and therefore every differential with image in $E_r^{0,4}$ vanishes. In particular the differential $d^5_{5,0}=0$, and therefore $E_{\infty}^{0,4} \cong E_2^{0,4} \cong \Omega^{\TOPSpin}_4 \cong \Z$.
We deduce that \[\Omega^{\TOPSpin}_4(S^1\times\CP^\infty)\cong E_{\infty}^{0,4}\oplus \ker(d_3^{4,0} \colon E_3^{4,0} \to E_3^{1,2}) \cong\Z\oplus \ker(d_3^{4,0} \colon 2\Z \to \Z/2).\]

The obstruction in $E_3^{4,0}\cong \Z\cong \Omega_4^{\TOPSpin}$ determined by $W$ is given by the signature of $W$ (divided by 8), which vanishes by Lemma~\ref{lem:CPinftycompatible}. The obstruction in $E_{\infty}^{4,0}\subseteq \Omega^{\TOPSpin}_4(S^1\times\CP^\infty)$ determined by $W$ is given by $\eta_*([W])\in H_4(\CP^\infty;\Z)$. Using Lemma~\ref{lem:CPinftycompatible}, we assumed that the map $\eta$ is such that this obstruction vanishes.
In addition, during the proof of Lemma~\ref{lem:CPinftycompatible} we showed that one can arrange for $\eta_*([W])$ to take any even value.  It follows that $E_{\infty}^{4,0} \subseteq E_2^{4,0} \cong \Z$ is the subgroup $2\Z$, and so in fact we must have $d_3^{4,0} =0$ and $\ker d_3^{4,0} = E_3^{4,0} \cong 2\Z$. We conclude that $\Omega^{\TOPSpin}_4(S^1\times\CP^\infty) \cong \Z \oplus 2\Z$, and that $(W,\mathfrak{s}\times\lambda\times\eta)$ is the trivial element of this group.

Hence $(W, \mathfrak{s}\times\lambda\times \eta)$ vanishes in the group $\Omega_4(B(n),p_n)$. By Lemma~\ref{lem:sanitycheck}, this implies that $(W_0,\overline{\nu}_{W_0})$ is cobordant to $(W_1,\overline{\nu}_{W_1})$ rel.~boundary. Choose such a bordism $(Z,\Xi)$. By Theorem~\ref{thm:kreckmain} this determines a surgery obstruction
$\theta(Z,\Xi)\in \displaystyle{L^{s,\tau}_5(\Z[\Z])}$, and if this obstruction vanishes then $(W_0,\overline{\nu}_{W_0})$ is $s$-cobordant to $(W_1,\overline{\nu}_{W_1})$ rel.~boundary.
As $\Wh(\Z)=0$, the exact sequence \eqref{eq:kreck} shows that the obstruction in fact lies in the subgroup $\displaystyle{L^{s}_5(\Z[\Z])}\subseteq \displaystyle{L^{s,\tau}_5(\Z[\Z])}$.
This obstruction group is well-known to be $\displaystyle{L_5^{s}(\Z[\Z])}\cong \displaystyle{L_4^{h}(\Z)}\cong 8\Z$, given by the signature of a closed codimension 1 submanifold that meets an embedded loop $\gamma\subseteq Z$ algebraically once, where $\gamma\subseteq Z$ is a loop such that $\Xi(\gamma)$ generates $\pi_1(B(0))\cong\Z$. This signature may be nonzero. To kill the signature, we take the internal $S^1$-sum between $Z$ and copies of $S^1\times E_8$, or $S^1\times -E_8$ as appropriate, identifying copies of $\gamma\subseteq Z$ in the former with representatives of the $S^1$ factors in the latter. After modifying $Z$ in this way, the obstruction in $\displaystyle{L^s_5(\Z[\Z])}$ vanishes and so the desired $s$-cobordism exists by Theorem~\ref{thm:kreckmain}. As $\Z$ is a good group, the Freedman-Quinn $s$-cobordism theorem~\cite[Theorem~7.1A]{FQ} implies there exists a homeomorphism $F\colon W_0\cong W_1$ restricting to the identity on $\partial X$ and to $g_1^{-1}\circ g_0\colon L_0\to L_1$ on this boundary component.
\end{proof}

\section{Proof of Theorem~\ref{theoremA}}\label{section:the-proof}

We now record two short lemmas before moving on to the proof of Theorem~\ref{theoremA}.

\begin{lemma}\label{lem:fix}
Let $X\simeq S^2$ be a compact 4-manifold with intersection form $(n)$, where $n \in \Z$.  Suppose $F\colon X\to X$ is an orientation preserving homeomorphism, restricting to the identity on $\partial X$, and such that the map $F_*\colon H_2(X)\to H_2(X)$ is multiplication by $-1$. Then $n\in\{\pm 1, \pm 2\}$.
\end{lemma}

\begin{proof}
By the universal coefficient theorem we have $F^*\colon H^2(X)\to H^2(X)$ is also multiplication by $-1$. Consider the commuting square with downwards maps given by $F^*$ and horizontal maps induced by inclusion
\[
\begin{tikzcd}
H^2(X)\ar[r]\ar[d, "-1"]
&H^2(\partial X)\ar[d, "1"]
\\
H^2(X)\ar[r]
&H^2(\partial X).
\end{tikzcd}
\]
It is straightforward to identify the horizontal maps with the quotient map $\Z\to\Z/n$. The commutativity of the square implies~$1\equiv-1\mod{n}$, which holds if and only if  $n \in \{\pm 1,\pm 2\}$.
\end{proof}

\begin{lemma}\label{lem:fix2}
Let $X\simeq S^2$ be a compact 4-manifold. Let $S\subseteq X$ be a spine with self-intersection~$n$, where $n\in\{\pm1, \pm 2\}$. Then there is an orientation preserving homeomorphism $F\colon X\to X$, fixing $\partial X$ pointwise, and such that $F|_S\colon S\to S$ is an orientation reversing homeomorphism.
\end{lemma}

\begin{proof}
Write $(D^2 \wt{\times}_n S^2, L(n,1))$ for the Euler number $n$ disc bundle over $S^2$. Choose an orientation preserving homeomorphism $G\colon \overline{\nu}(S)\cong D^2 \wt{\times}_n S^2$ that preserves the 0-section. Apply the orientation preserving map $\tau$ from Definition~\ref{defn:tau} to the disc bundle $D^2 \wt{\times}_n S^2$. When $n=\pm1 $, we have $L(n,1)\cong S^3$ and so any orientation preserving self-homeomorphism of $L(n,1)$ is isotopic to the identity. When $n=\pm 2$, we have $L(n,1)\cong \RP^3$, so by \cite[Th\'{e}or\`{e}me 3(e)]{MR710104} every orientation preserving self-homeomorphism of $L(n,1)$ is isotopic to the identity. Hence in either case, in a boundary collar of $L(n,1)$, we may insert an isotopy to obtain an orientation preserving self-homeomorphism $\widetilde{\tau}$ of $D^2 \wt{\times}_n S^2$, that fixes the boundary pointwise, that fixes the base~$S^2$ setwise, and is orientation reversing on the base~$S^2$ (Lemma~\ref{lem:taucalculation}). Extend the map $G^{-1}\circ\widetilde{\tau}\circ G$ from $\overline{\nu}S$ to $X$ by the identity, to obtain the desired~$F$.
\end{proof}

Now we have all the ingredients we need to complete the proof of Theorem~\ref{theoremA}, whose statement we recall for the convenience of the reader.

\begin{theorem}
  Let $X\simeq S^2$ be a compact 4-manifold. Let $S_0$ and $S_1$ be simple spines. If $[S_0]=[S_1] \in H_2(X)$ then $S_0$ and $S_1$ are ambiently isotopic via an isotopy that restricts to the identity on~$\partial X$.

\end{theorem}

\begin{proof}
By Proposition~\ref{prop:exists2spherehomeo} there is a homeomorphism $X \sm \nu S_0 \to X \sm \nu S_1$ restricting to the identity on $\partial X$ and to a homeomorphism $\partial \ol{\nu} S_0 \to \partial \ol{\nu} S_1$ that is identified, for suitable boundary parameterisations, with a self-homeomorphism of $S^1 \wt{\times}_n S^2 \cong L(n,1)$ that extends to an orientation preserving bundle isomorphism of the disc bundle $D^2 \wt{\times}_n S^2$.  Choose such an extension to extend the homeomorphism $\partial \ol{\nu} S_0 \to \partial \ol{\nu} S_1$ to an orientation preserving homeomorphism $\ol{\nu} S_0 \to \ol{\nu} S_1$ that preserves the 0-sections, i.e.~that maps $S_0$ to $S_1$
 %via an (orientation preserving)
 via a homeomorphism.  We obtain an orientation preserving homeomorphism $F \colon X \to X$, fixing $\partial X$ pointwise and sending $S_0$ to $S_1$ via a homeomorphism (which is potentially orientation reversing).

Suppose $[S_0]=[S_1] \in H_2(X)$. The map $F_*$ on $H_2(X)$ is an isomorphism, so it is multiplication by $\pm1$. By Lemma~\ref{lem:fix}, when $n\not\in\{\pm 1,\pm 2\}$ we must have that $F_*$ is the identity map on $H_2(X)$, so that $F_*([S_0])=[S_0]=[S_1]$, and so $F\colon S_0\to S_1$ is orientation preserving. When $n\in\{\pm 1,\pm 2\}$, it is possible that $F_*$ is multiplication by $-1$ on $H_2$. If this is the case, modify $F$ by postcomposing with the map from Lemma~\ref{lem:fix2}, so that $F_*$ becomes the identity on $H_2$ and hence $F_*([S_0])=[S_0]=[S_1]$, and so $F\colon S_0\to S_1$ is orientation preserving. Now as $F_*$ is the identity on $H_2(X)$, by Theorem~\ref{corollaryC}~\cite[Corollary~C]{OP-mcgs}, $F$ is topologically isotopic rel.\ boundary to the identity on $X$. The resulting isotopy is a rel.\ boundary ambient isotopy between $S_1$ and~$S_0$. This completes the proof of the theorem.
\end{proof}

\begin{remark}
Suppose $X\simeq S^2$ is a compact 4-manifold with simple spines $S_0$ and $S_1$ such that $[S_0]=- [S_1]\in H_2(X)$. Then it is impossible that the spines are related by an ambient isotopy, because ambient isotopy preserves spine orientation. By Lemma~\ref{lem:fix}, when $n\not\in\{\pm1, \pm2\}$, it cannot even be the case that there is a homeomorphism of pairs $F\colon (X,S_0)\to (X,S_1)$, fixing~$\partial X$ pointwise, that is both orientation preserving from $X$ to $X$, and from $S_0$ to $S_1$. The anonymous referee asked whether such a homeomorphism \emph{does} exist in the cases $n\in\{\pm 1,\pm2\}$. We can answer this question affirmatively. To see this, follow the proof of Theorem~\ref{theoremA} to the end of the first paragraph. The map $F$ thus constructed is either orientation preserving from~$S_0$ to $S_1$ or not. If it is not, then postcompose $F$ with the map from Lemma~\ref{lem:fix2}, to produce the desired map.
\end{remark}

\bibliographystyle{alpha}
\bibliography{spines}
\end{document}